\documentclass{amsart}
\usepackage{amscd,amssymb,amsopn,amsmath,amsthm,graphics,amsfonts,accents,enumerate,verbatim,calc}
\usepackage[dvips]{graphicx}
\usepackage[pagebackref,colorlinks=true,linkcolor=red,citecolor=blue]{hyperref}
\usepackage[all]{xy}
\usepackage{cite}
\usepackage{tikz}
\usepackage{tikz-cd}
\usepackage{mathrsfs}

\calclayout

\newcommand{\rt}{\rightarrow}
\newcommand{\lrt}{\longrightarrow}

\newcommand{\va}{\vartheta}
\newcommand{\st}{\stackrel}

\newcommand{\la}{\lambda}
\newcommand{\La}{\Lambda}
\newcommand{\Ga}{\Gamma}
\newcommand{\lan}{\langle}
\newcommand{\ran}{\rangle}

\newcommand{\CA}{\mathcal{A} }

\newcommand{\DD}{\mathcal{D} }

\newcommand{\CF}{\mathcal{F} }

\newcommand{\CH}{\mathcal{H}}
\newcommand{\CI}{\mathcal{I} }

\newcommand{\CK}{\mathcal{K} }
\newcommand{\CL}{\mathcal{L} }

\newcommand{\CR}{\mathcal{R} }
\newcommand{\CS}{\mathcal{S} }

\newcommand{\CU}{\mathcal{U}}
\newcommand{\CV}{\mathcal{V}}

\newcommand{\CX}{\mathcal{X} }
\newcommand{\CY}{\mathcal{Y} }

\newcommand{\SU}{\mathscr{U}}
\newcommand{\SV}{\mathscr{V}}

\newcommand{\Mod}{{\rm{Mod\mbox{-}}}}

\newcommand{\mmod}{{\rm{{mod\mbox{-}}}}}

\newcommand{\op}{{\rm{op}}}

\newcommand{\Ker}{{\rm{Ker}}}

\newcommand{\Tor}{{\rm{Tor}}}
\newcommand{\Hom}{{\rm{Hom}}}
\newcommand{\Ext}{{\rm{Ext}}}

\theoremstyle{plain}
\newtheorem{theorem}{Theorem}[section]
\newtheorem{corollary}[theorem]{Corollary}
\newtheorem{lemma}[theorem]{Lemma}

\newtheorem{proposition}[theorem]{Proposition}

\newtheorem{notation}[theorem]{Notation}

\theoremstyle{definition}
\newtheorem{definition}[theorem]{Definition}

\newtheorem{construction}[theorem]{Construction}

\newtheorem{remark}[theorem]{Remark}

\newtheorem{s}[theorem]{}

\theoremstyle{plain}

\theoremstyle{definition}

\numberwithin{equation}{section}

\begin{document}

\title[Monomorphism category of large modules]{On the monomorphism category of large modules}

\author[R. Hafezi, J. Asadollahi, R. Vahed and Y. Zhang]
{Rasool Hafezi, Javad Asadollahi, Razieh Vahed and Yi Zhang }

\address{School of Mathematics and Statistics, Nanjing University of Information Science \& Technology, Nanjing, Jiangsu 210044, P.\,R. China}
\email{hafezi@nuist.edu.cn}

\address{Department of Pure Mathematics, Faculty of Mathematics and Statistics, University of Isfahan, P.O.Box: 81746-73441, Isfahan, Iran}
\email{asadollahi@sci.ui.ac.ir, asadollahi@ipm.ir}

\address{Department of Mathematics, Khansar Campus, University of Isfahan, Iran and, School of Mathematics, Institute for Research in Fundamental Science (IPM), P.O.Box: 19395-5746, Tehran, Iran}
\email{r.vahed@khc.ui.ac.ir, vahed@ipm.ir}

\address{School of Mathematics and Statistics, Nanjing University of Information Science \& Technology, Nanjing, Jiangsu 210044, P.\,R. China}
\email{zhangy2016@nuist.edu.cn}



\begin{abstract}
Let $R$ be an associative ring with identity. This paper investigates the structure of the monomorphism category of large $R$-modules and establishes connections with the category of contravariant functors defined on finitely presented $R$-modules. Several equivalences and dualities will be presented. Our results highlight the role of pure-injective modules in studying the homological properties of functor categories.
\end{abstract}

\maketitle

\tableofcontents

\section{Introduction}
Let $\CA$ be an abelian category. The monomorphism category of $\CA$, denoted by $\CS(\CA)$, is itself an abelian category whose objects are all morphisms in $\CA$. Morphisms are defined by commutative diagrams. Monomorphism categories were first introduced by G. Birkhoff in 1934, under the name `submodule categories'. Over the last two decades, the significance of monomorphism categories has been explored from various perspectives, as they provide a framework for tackling open problems in linear algebra using results and methods from homological algebra, combinatorics, and geometry. An important application in data science is that the systems of complexes that lead to persistence homology form a monomorphism category.

Birkhoff \cite{Bi} studied the monomorphism category of $\La$ for the case where $\La = \mathbb{Z}/(q^p)$, with $p \geq 2$ and $q$ being prime numbers, and $\mathbb{Z}$ representing the ring of integers. Simson \cite{Si} investigated the case where the algebra is $k[x]/(x^n)$, with $k$ as a field, and examined the tame-wild dichotomy for the monomorphism category $\CS(\mmod k[x]/(x^n))$ in relation to the parameter $n$.

More generally, the monomorphism category of an abelian category has been the subject of study, which has potential applications in algebraic geometry.
For example, when the abelian category is represented by the graded module category over the graded algebra $k[x]/(x^n)$, where $\text{deg } x = 1$, the corresponding monomorphism category is denoted by $\CS(\tilde{n})$. This category is a Frobenius exact category, and thus its stable category is triangulated \cite{Chen}. An interesting result by Kussin, Lenzing, and Meltzer demonstrates that the stable category of $\CS(\tilde{n})$ is triangle-equivalent to the stable category of vector bundles on the weighted projective lines of type $(2, 3, n)$.

More recently, Ringel and Schmidmeier investigated the representation type of $\CS(\mmod\La)$ and explored Auslander-Reiten theory for certain submodule categories \cite{RS1}. They provided a specific method for computing the (relative) Auslander-Reiten translation in $\CS(\mmod\La)$ \cite{RS2}. Additionally, Ringel and Zhang \cite{RZ} established a significant connection between the monomorphism categories of factor algebras of $k[x]$ with respect to a power of $x$ and preprojective algebras of type $A$. It is important to note that preprojective algebras, introduced by Gelfand and Ponomarev, have been of significant interest as they serve as vital tools in the study of quantum groups and cluster algebras.

This connection, as discussed in \cite{Ha2}, was generalized to encompass a broader spectrum involving certain monomorphism categories and relative stable Auslander algebras. Eiríksson \cite{E} examined the results of Ringel and Zhang in a more general context of representation-finite self-injective Artin algebras \cite[\S 4]{E}, and a higher homological algebra version of these results is considered in \cite{AHS}. Moreover, this theory has been further generalized by Xiong, Zhang, and Zhang \cite{XZZ} to the category $\CS_n(\La)$, which consists of all $(n - 1)$-sequences $X_{n-1} \st{f_{n-1}}\rightarrow X_{n-2} \st{f_{n-2}}\rightarrow \ldots \rightarrow X_1 \st{f_1}\rightarrow X_0$, where every $f_i$ is a monomorphism. Additionally, monomorphism categories related to arbitrary species have also been studied in \cite{GKKP}.

The main aim of this paper is to explore the results presented in \cite{RZ} within a broader framework of $ R\text{-}{\rm Mod}$, encompassing all modules, not just finitely presented ones. In this context, pure-injective modules play a crucial role. The introduction of \cite{K2} provides a good account on the importance of passage from small mod to large Mod. This transition from $ R\text{-}{\rm mod}$ to $ R\text{-}{\rm Mod}$ represents a significant paradigm shift. It inherently involves new concepts and techniques

Building upon the functors introduced by Ringel and Zhang \cite{RZ}, we propose a functor from the monomorphism category of pure-injective modules, denoted by $\CS(R\text{-}{\rm Pinj})$ to the category of all (not necessarily finitely presented) contravariant functors on the category of finitely presented $ R$-modules, denoted as $ ((R\text{-}{\rm mod})^{\rm op}, \mathcal{A}{\rm b})$.

To achieve this, we leverage the structure of flat modules and the concept of special flat resolutions in $((R\text{-}{\rm mod})^{\rm op}, \mathcal{A}{\rm b})$, as extensively studied by Crawley-Boevey \cite{CB} and Herzog \cite{He}. Herzog demonstrated that every object in this category has a special flat resolution of length at most two. Our first main result, Theorem \ref{Thm.Psi}, establishes a full, dense, and objective functor from $\CS(R\mbox{-}{\rm Pinj})$ to $ ((R\text{-}\underline{{\rm mod}})^{\rm op}, \mathcal{A}{\rm b})$, where $\underline{{\rm mod}}$ denotes the stable category.

By combining these two functors, we establish a duality $\varrho$ between $ R\text{-}\underline{{\rm mod}}$ and $\underline{{\rm mod}}\mbox{-}R$. It is shown in Proposition \ref{ConnTranspose} that if $R$ is an Artinian ring with Morita self-duality (for example, if $R$ is an Artin algebra), this duality corresponds to Auslander's transpose, denoted as {\rm Tr}.

We then focus on the Quasi-Frobenius rings of finite representation type. Theorem \ref{Theta-Theorem} establishes that over such rings, there exists a full, dense, and objective functor from $\CS(R\text{-}{\rm Mod})$ to $((R\text{-}\underline{{\rm mod}})^{\rm op}, \mathcal{A}\text{{b}})$. In the covariant version, Theorem \ref{SeconTheoemcovariant} demonstrates that for a left pure-semisimple and Quasi-Frobenius ring $R$, there exists a full, dense, and objective functor from $\CF(R\text{-}{\rm Mod})$ to $(\underline{{\rm mod}}\text{-}R, \mathcal{A}\text{{b}})$.

The final section of the paper is dedicated to presenting an application.

\begin{notation}
Throughout the paper, $R$  is an associative ring with identity. By $R$-module we always mean left $R$-module, unless otherwise specified. The category of $R$-modules will be denoted by $R\text{-}{\rm Mod}$.  We need the following subcategories of $R\text{-}{\rm Mod}$. All subcategories are assumed to be full.
\[\begin{array}{lll}
R\text{-}{\rm mod} & =  & {\rm Finitely \ presented} \ R\mbox{-}{\rm modules}\\
R\text{-}{\rm Prj} & = & {\rm Projective} \ R\mbox{-}{\rm modules}\\
R\text{-}{\rm prj} & = & {\rm Finitely \ generated, \ projective} \ R\mbox{-}{\rm modules}\\
R\text{-}{\rm Flat} & =  & {\rm Flat} \ R\mbox{-}{\rm modules}\\
R\text{-}{\rm Pprj} & = & {\rm Pure \ projective} \ R\mbox{-}{\rm modules}\\
R\text{-}{\rm Pinj} & = & {\rm Pure \ injective} \ R\mbox{-}{\rm modules}.
\end{array}\]
\end{notation}

\section{Special functors}\label{Sub. flact functors}
In this section, we recall some facts that we need throughout the paper. Consider the following known Grothendieck categories
\[\begin{array}{lll}
((R\text{-}{\rm mod})^{\rm op}, \mathcal{A}{\rm b}) & = & {\rm Additive \ contravariant \ functors} \ F: R\text{-}{\rm mod} \lrt \CA{\rm b},\\
(R^{\rm op}\text{-}{\rm mod}, \mathcal{A}{\rm b}) & = & {\rm Additive \ covariant \ functors} \ F: {\rm mod}\text{-}R \lrt \CA{\rm b}.
\end{array}\]
where $\CA{\rm b}$ denotes the category of abelian groups, see e.g. \cite{JL} and \cite{He}.

One of the key tools in this setting is Yoneda's Lemma, see e.g. \cite[Proposition IV.7.3]{S}. It asserts that for every $A \in R\text{-}{\rm mod}$  and every $F$ in $((R\text{-}{\rm mod})^{\rm op}, \mathcal{A}{\rm b})$, there is an isomorphism
\[Y(A,F): {\rm Nat}(( - , A), F) \rt F(A),\]
of abelian groups, natural in both $A$ and $F$. We will use this lemma and its covariant version freely throughout the paper without additional references.

\subsection{Contravariant version}
Throughout the paper, $\Hom_R( - , - )$ will be denoted by $( - , - )$. Let $F $ be an object of $((R\text{-}{\rm mod})^{\rm op}, \mathcal{A}{\rm b})$. Then $F$ is called representable if there is a finitely presented module $A$ such that $F\simeq (-, A)$. According to Yoneda Lemma every representable functor is projective. Moreover, it is known \cite[Theorem B.10]{JL} that a functor $F$ in $((R\text{-}{\rm mod})^{\rm op}, \mathcal{A}{\rm b})$ is projective if and only if $F \cong ( - , M)$ for some pure-projective $R$-module $M$. Note that an $R$-module is pure-projective
if and only if it is a direct sum of finitely presented modules, see e.g. \cite[Proposition 4.3.72]{P}.

\s A functor $F$ in $((R\text{-}{\rm mod})^{\rm op}, \mathcal{A}{\rm b})$ is called finitely presented if there is an exact sequence
\[ ( - ,M_1) \lrt ( - , M_0) \lrt F \lrt 0,\]
where  $M_0 , M_1 \in R\text{-}{\rm mod}$. Auslander \cite{Au} showed that if $R$ is left coherent, then the subcategory ${\rm fp}((R\text{-}{\rm mod})^{\rm op}, \mathcal{A}{\rm b})$ of $((R\text{-}{\rm mod})^{\rm op}, \mathcal{A}{\rm b})$, consisting of all finitely presented contravariant functors, is an abelian category.

\s An object $F \in  ((R\text{-}{\rm mod})^{\rm op}, \mathcal{A}{\rm b})$ is flat if it is isomorphic to a direct limit of finitely generated projective functors. By \cite[Theorem 1.4]{CB} every flat functor is isomorphic to $( - , A)$, for some $R$-module $A$. It is shown by Herzog \cite{He} that every object in the category $((R\text{-}{\rm mod})^{\rm op}, \mathcal{A}{\rm b})$ is of flat dimension at most $2$, that is, for every $F \in  ((R\text{-}{\rm mod})^{\rm op}, \mathcal{A}{\rm b})$  there exists an exact sequence
\begin{equation}\label{eq. specialprecover}
\begin{tikzcd}
 0 \rar & ( - , M_2) \rar{( - ,g)} & ( - , M_1) \rar{( - , f)} & ( - , M_0) \rar{d} & F \rar & 0
\end{tikzcd}
\end{equation}
where $M_0, M_1$ and $M_2$ are $R$-modules.

\s There is a full and faithful functor
\begin{equation}\label{Upsilonfunc}
\Upsilon:R\mbox{-}{\rm Mod}\rt ((R\text{-}{\rm mod})^{\rm op}, \mathcal{A}{\rm b}),
\end{equation}
given by $M\mapsto ( - , M)$. By \cite[Theorem 1.4]{CB}, this functor induces an equivalence between $R\mbox{-}{\rm Mod}$ and
${\rm Flat}((R\text{-}{\rm mod})^{\rm op}, \mathcal{A}{\rm b})$, the subcategory of flat functors in $((R\text{-}{\rm mod})^{\rm op}, \mathcal{A}{\rm b})$. Its restriction to $R\mbox{-}{\rm mod}$, i.e.
\begin{equation*}\label{equ. full.faith}
\Upsilon|: R\mbox{-}{\rm mod}\rt ((R\text{-}{\rm mod})^{\rm op}, \mathcal{A}{\rm b})
\end{equation*}
induces an equivalence between $R\mbox{-}{\rm mod}$ and
${\rm prj}((R\text{-}{\rm mod})^{\rm op}, \mathcal{A}{\rm b})$, the subcategory of finitely generated projective functors in $((R\text{-}{\rm mod})^{\rm op}, \mathcal{A}{\rm b})$ \cite[Corollary 7.4]{S}.

\s Let $F$ be an object of  $((R\text{-}{\rm mod})^{\rm op}, \mathcal{A}{\rm b})$. A morphism $\varphi: ( - ,M) \rt F$ is called a flat precover of $F$ if for any flat object $( - ,N)$, the induced morphism $$\Hom(( - ,N), ( - ,M)) \lrt \Hom(( - ,N), F)$$ is an epimorphism of abelian groups. It is called a flat cover of $F$ if any $R$-endomorphism $h: M \rt M$ with the property that $\varphi \circ ( -,h)=\varphi$,  is an automorphism of $M$. The existence of flat covers in $((R\text{-}{\rm mod})^{\rm op}, \mathcal{A}{\rm b})$  follows from \cite[Proposition 2.4]{SD}. A flat precover is called special if it has a cotorsion kernel. Recall that a functor $F \in ((R\text{-}{\rm mod})^{\rm op}, \mathcal{A}{\rm b})$ is cotorsion if $\Ext^1((-, M), F )=0$, where $M \in {R}\mbox{-}{\rm Mod}$. By \cite[Theorem 4]{He}, a flat functor  $(-, M) \in ((R\text{-}{\rm mod})^{\rm op}, \mathcal{A}{\rm b})$ is cotorsion if and only if $M$ is a pure-injective module.

\begin{remark}
Given an object $F$ in $((R\text{-}{\rm mod})^{\rm op}, \mathcal{A}{\rm b})$, a flat resolution \eqref{eq. specialprecover} of $F$ is called special if both morphisms $d: ( - , M_0)\rt F$ and $( - , f):( - , M_1)\rt {\rm Im}( - , f)$ are special flat precovers. By Proposition 7 of \cite{He}, every object $F$ in $((R\text{-}{\rm mod})^{\rm op}, \mathcal{A}{\rm b})$ admits a special flat resolution. Additionally, a flat resolution of F as given in \eqref{eq. specialprecover} is special if and only if $M_1$ and $M_2$ are pure-injectives \cite[Proposition 5]{He}.
\end{remark}

Following the proof of \cite[Proposition 7]{He}, we will explain how to construct a special flat resolution of a functor $F$ in $((R\text{-}{\rm mod})^{\rm op}, \mathcal{A}{\rm b})$. In the construction, we tacitly used Proposition 3 of \cite{He} that states in any short exact sequence
\begin{equation*}
\begin{tikzcd}
0 \rar & ( - ,M) \rar & G \rar & ( - , K) \rar & 0,
\end{tikzcd}
\end{equation*}
in $((R\text{-}{\rm mod})^{\rm op}, \mathcal{A}{\rm b})$, $G$ is flat and hence is of the form $( - , N)$, for some $R$-module $N$.

\begin{construction}\label{ConsFlatRes}
Let $F$ be a functor in $((R\text{-}{\rm mod})^{\rm op}, \mathcal{A}{\rm b})$. Consider a flat resolution
\begin{equation*}
\begin{tikzcd}
 0 \rar & ( - , M_2) \rar{( - ,g)} & ( - , M_1) \rar{( - , f)} & ( - , M_0) \rar & F \rar & 0
\end{tikzcd}
\end{equation*}
where $M_0, M_1, M_2 \in R\mbox{-}{\rm Mod}$. Set $G={\rm Im}(-, f).$  Let $e: M_2 \rt PE(M_2) $ be the pure-injective envelope of $M_2$. The pushout construction yields the commutative diagram
\begin{equation}\label{digarm1Con}
\xymatrix{	 & 0 \ar[d]  & 0 \ar[d]
 &   & \\0 \ar[r] & (-, M_2)\ar[d] \ar[r] & (-, M_1) \ar[d]
\ar[r] &  G\ar@{=}[d] \ar[r] & 0\\
	0 \ar[r] & (-, PE(M_2)) \ar[d] \ar[r] & (-, N)
	\ar[d] \ar[r] & G  \ar[r] & 0\\
0 \ar[r]  & (-, PE(M_2)/M_2) \ar[d] \ar@{=}[r] & (-, PE(M_2)/M_2)\ar[d]
	&    & \\
& 0  & 0  &  & }
\end{equation}

Next, let $\ell: N \rt PE(N)$ be a pure-injective envelop of $N$, and consider the pushout diagram

\begin{equation}\label{diagram2Con}
\xymatrix{	 & 0 \ar[d]  & 0 \ar[d]
 &  0\ar[d] & \\0 \ar[r] & (-, PE(M_2))\ar@{=}[d] \ar[r] & (-, N) \ar[d]
\ar[r] &  G\ar[d] \ar[r] & 0\\
	0 \ar[r] & (-, PE(M_2))  \ar[r] & (-, PE(N))
	\ar[d] \ar[r] & G'  \ar[r]\ar[d] & 0\\
  &   & \ar[d]
( - , PE(N)/N) \ar@{=}[r] & ( - , PE(N)/N)\ar[d]   & \\
&   & 0  & 0 & }
\end{equation}

Now, we apply the last column of the above diagram, to get the pushout diagram
\begin{equation}\label{diagram3Con}
\xymatrix{	 & 0 \ar[d]  & 0 \ar[d]
 &   & \\0 \ar[r] & G\ar[d] \ar[r] & (-, M_0) \ar[d]
\ar[r] &  F\ar@{=}[d] \ar[r] & 0\\
	0 \ar[r] & G' \ar[d] \ar[r] & (-, L)
	\ar[d] \ar[r] & F \ar[r] & 0\\
0 \ar[r]  & (-, PE(N)/N) \ar[d] \ar@{=}[r] & (-, PE(N)/N)\ar[d]
	&    & \\
& 0  & 0  &  & }
\end{equation}
Finally, the last two pushout diagrams above induce the exact sequence
\begin{equation}\label{SequCon2}
\begin{tikzcd}
 0 \rar & (-, PE(M_2)) \rar & (-, PE(N)) \rar & (-, L) \rar{\pi} & F \rar & 0
\end{tikzcd}
\end{equation}
which is the desired special flat resolution for $F$.
\end{construction}

\subsection{Covariant version}
In parallel with contravariant functors, in this subsection, we briefly recall some results for the functor category $(\mmod R, \CA{\rm b}) = (R^{\rm op}\text{-}{\rm mod}, \mathcal{A}{\rm b})$ of covariant functors on $\mmod R$.  Throughout the paper, $ - \otimes_R - $ will be denoted by $( - \otimes - )$. There exists a full embedding
\begin{equation}\label{functor t}
t: R\mbox{-}{\rm Mod} \lrt (\mmod R, \CA{\rm b})
\end{equation}
that associates each left $R$-module $M$ with the functor $- \otimes M$, see \cite[Chapter 1]{K2}.

\s Let ${\rm fp}(\mmod R, \CA{\rm b})$ denote the full subcategory of $(\mmod R, \CA{\rm b})$ consisting of all finitely presented objects. Auslander \cite[§III.2]{Au2} showed that ${\rm fp}(\mmod R, \CA{\rm b})$ is abelian without any hypotheses on the ring $R$. Moreover, Auslander \cite{Au3} and Gruson and Jensen \cite{GJ} independently showed that there exists a self-duality on ${\rm fp}(\mmod R, \CA{\rm b})$.

\s A functor $F$ in $(\mmod R, \mathcal{A}{\rm b})$  is called fp-injective if $\Ext^1_R(G,F)=0$, for all finitely presented functors $G$ in ${\rm fp}(\mmod R, \mathcal{A}{\rm b})$. By \cite[Lemma 1.4]{C} a functor $F$ in $(\mmod R, \CA{\rm b})$ is fp-injective if and only if $F \cong  - \otimes M$, for some (left) $R$-module $M$. Hence, the embedding \eqref{functor t} characterizes fp-injective objects in $(\mmod R, \mathcal{A}{\rm b})$ as covariant functors of the form  $-\otimes M$, see also \cite[Theorem B.16]{JL}.

\s Gruson and Jensen \cite{GJ} characterized the injective objects of  $(\mmod R, \CA{\rm b})$ as those functors that are isomorphic to $ - \otimes M$, for some pure-injective $R$-module $M$. For  every $F$ in $(\mmod R, \mathcal{A}{\rm b})$, there is an injective copresentation
\begin{equation*}
\begin{tikzcd}
0 \rar & F \rar  & - \otimes E^0 \rar{\varphi}&  -\otimes E^1
\end{tikzcd}
\end{equation*}
where $E^0$ and $E^1$ are pure-injectives, see \cite[Page 379]{He}. Moreover, since the functor $t$ of \eqref{functor t} is full and faithful, there is a morphism $f: E^0 \rt E^1$ of $R$-modules such that $\varphi= - \otimes f$. If we set $L:={\rm Cok} f$, then the above injective copresentation of $F$ can be completed as in the following
\begin{equation}\label{eq. recolCOver1}
\begin{tikzcd}
0 \rar & F \rar & -\otimes E^0 \rar{-\otimes f} & - \otimes E^1  \rar{-\otimes g} & - \otimes L \rar & 0,
\end{tikzcd}
\end{equation}
where $g:E^1 \rt L$. The last term is not necessarily an injective functor; however, it is an fp-injective object. This, in turn, shows that every functor $F$ in $(\mmod R, \mathcal{A}{\rm b})$ admits an fp-injective coresolution of length at most $2$.

\section{Two recollements}
In this section, we recall two recollements from \cite{AAHV} and apply results from \cite{DR2} to study the functors that appear in these recollements. Although we only need some of the functors, we will briefly investigate all of them for the sake of completeness.

\subsection{Contravariant version}
By \cite[subsection 3.1]{AAHV}, the functor $\Upsilon$ introduced in \eqref{Upsilonfunc} embeds into the recollement
\begin{equation}\label{RecollContr}
\xymatrix@C=0.5cm{ ((R\text{-}\underline{{\rm mod}})^{\rm op}, \mathcal{A}{\rm b}) \ar[rrr]^{i }   &&&  ((R\text{-}{\rm mod})^{\rm op}, \mathcal{A}{\rm b}) \ar[rrr]^{\nu} \ar@/^1.5pc/[lll]^{i_{\rho} } \ar@/_1.5pc/[lll]_{i_{\la}} &&& R\mbox{-}{\rm Mod},\ar@/^1.5pc/[lll]^{\Upsilon} \ar@/_1.5pc/[lll]_{L} }
\end{equation}
where $i:((R\text{-}\underline{{\rm mod}})^{\rm op}, \mathcal{A}{\rm b}) \rt ((R\text{-}{\rm mod})^{\rm op}, \mathcal{A}{\rm b})$ is the inclusion functor.

In the following we study the other functors appearing in this recollement, explicitly.

\s The functor $\nu : ((R\text{-}{\rm mod})^{\rm op}, \mathcal{A}{\rm b}) \rt R\mbox{-}{\rm Mod}$ is defined in \cite[Subsection 3.1]{AAHV}. Let $F$ be an object of $((R\text{-}{\rm mod})^{\rm op}, \mathcal{A}{\rm b})$ with a special flat resolution
\begin{equation}\label{eq. firstRecoll}
\begin{tikzcd}
 0 \rar & ( - , E_1) \rar{( - ,g)} & ( - , E_0) \rar{( - , f)} & ( - , L) \rar & F \rar & 0.
\end{tikzcd}
\end{equation}
Then $\nu(F):={\rm Cok} f.$ It acts naturally on morphisms.

\begin{remark}
It is easy to see that $(\nu, \Upsilon)$ is an adjoint pair, see for instance \cite[Theorem 3.1.7]{AAHV}. Since $\Upsilon$ is fully faithful, it is indeed a reflective adjoint pair, in the sense of \cite[Definition 2]{DR2}. Moreover, note that a projective functor $F$ is isomorphic to a functor of the form $( - , M)$, where $M$ is a pure-projective module. Hence, $M=\oplus_{i\in I}M_i$, where $M_i$, for $i\in I$, is a finitely presented module. Consequently $F\simeq \Upsilon(M)$. Hence the conditions of Theorem 10 of \cite{DR2} are satisfied and so it implies that the reflective adjoint pair $(\nu, \Upsilon)$ can be extended to a recollement, which is the above-mentioned recollement. In Section 3 of \cite{DR2}, a constructive way is given to obtain the functors in the above recollement. In the sequel, we will explain them in our setting.
\end{remark}

\s By \cite[Theorem 10]{DR2}, the functor $L$ is the zeroth left derived functor of $\Upsilon$. To compute it, let $M \in R\mbox{-}{\rm Mod}$ and take a projective presentation $P \st{f}\rt Q \rt M \rt 0$ of $M$ in $R\mbox{-}{\rm Mod}$. Then the functor $L(M)$ fits in the exact sequence
\begin{equation*}
\begin{tikzcd}
( - , P) \rar{( - , f)} & ( - , Q) \rar & L(M) \rar & 0.
\end{tikzcd}
\end{equation*}
The definition of $L$ on morphisms is defined in an obvious way by lifting a morphism to a morphism between the projective presentations.

\s \label{Subsect-ila} Let $F$ be a functor in $((R\text{-}{\rm mod})^{\rm op}, \mathcal{A}{\rm b})$ with a special flat resolution as given in \eqref{eq. firstRecoll}. Let $P \st{d}\rt Q \rt \nu(F) \rt 0$ be a projective presentation of $\nu(F)$ in $R\mbox{-}{\rm Mod}.$  So we have the commutative diagram
\[\xymatrix{P \ar[r]^{d} \ar@{.>}[d] & Q \ar[r] \ar@{.>}[d]^s & \nu(F) \ar[r] \ar@{=}[d]  & 0 \\	E_0 \ar[r]^{f} & L \ar[r] & \nu(F) \ar[r] & 0}\]
 in $R\mbox{-}{\rm Mod}.$  Applying the embedding functor $\Upsilon$ to the left square of the above  diagram leads to the commutative diagram
\[\xymatrix{( - , P) \ar[r]^{( - , d)} \ar[d] & ( - , Q) \ar[r] \ar[d]^{(-, s)} & L(\nu(F)) \ar[r] \ar@{.>}[d]^{\epsilon_F}  & 0 \\ ( - , E_0) \ar[r]^{(-, f)} & ( - , L) \ar[r] & F \ar[r] & 0}\]
 in $((R\text{-}{\rm mod})^{\rm op}, \mathcal{A}{\rm b})$. So we get the counit $\epsilon_F.$ In view of \cite[Lemma 6]{DR2}, $i_{\la}(F):={\rm Cok}\epsilon_F$, with the flat resolution
\begin{equation*}
\begin{tikzcd}
( - , E_0\oplus Q) \rar{(-, [f~~s])} & ( - , L) \rar & i_{\la}(F) \rar & 0,
\end{tikzcd}
\end{equation*}
that is not necessarily special.

\s Let $F$ be a functor in $((R\text{-}{\rm mod})^{\rm op}, \mathcal{A}{\rm b})$ with a special flat resolution as given in \eqref{eq. firstRecoll}. Then, there is the commutative diagram
\[ \xymatrix@C-0.5pc@R-0.5pc{ &&& 0 \ar[d] &0 \ar@{.>}[d] & \\ 0 \ar[r] & (-, E_1) \ar[r]^{(-, g)} \ar@{=}[d] & (-,E_0) \ar[r] \ar@{=}[d] & ( - , {\rm Im} f) \ar[r] \ar[d] & F_0 \ar@{.>}[d] \ar[r] & 0 \\
0\ar[r] & (-, E_1) \ar[r]^{(-, g)} & (-, E_0) \ar[r]^{(-, f)} & (-,L)\ar[d] \ar[r] & F \ar[r] \ar@{.>}[dl] & 0 \\
&&& (-, \nu(F)={\rm Cok}f) \ar[d] & & \\ &&& F_1 \ar[d] &&\\
&&& 0 && }\]
For more details see the proof of \cite[Proposition 3.1.5]{AAHV}. Now, based on \cite[Lemma 5]{DR2},  $i_{\rho}(F):=F_0.$

For the definition of $i_{\rho}$ on morphisms, let $\phi:F\rt G$ be a morphism. We can build a similar diagram as the above for the functor $G$ by taking a special flat resolution for the functor $G.$ The morphism $\phi$ can be lifted between special flat resolutions, which induces a morphism between the vertical exact sequences in the diagram associated to the functors $F$ and  $G.$ In particular, it gives  the morphism $i_{\rho}(\phi).$

\subsection{Covariant version}
We have the following recollement form \cite[Theorem 3.2.1]{AAHV}, see also \cite[Corollary 15]{DR2},

\[\xymatrix@C=0.5cm{ (\underline{{\rm mod}}\mbox{-}R, \mathcal{A}{\rm b}) \ar[rrr]^{j }   &&&  (\mmod R, \mathcal{A}{\rm b}) \ar[rrr]^{\vartheta} \ar@/^1.5pc/[lll]_{j_{\rho} } \ar@/_1.5pc/[lll]_{j_{\la}} &&& R \mbox{-}{\rm Mod},\ar@/^1.5pc/[lll]^{R_0} \ar@/_1.5pc/[lll]_{t} }\]
where $j$ is the inclusion and $ t$ is defined by $t(M)=-\otimes M$, as introduced in \eqref{functor t}.

 In the following, we will explicitly review the other functors appearing in this recollement.

\s The functor $\vartheta$  is simply the evaluation at ring, sending $F$ to $F(R_{R})$ \cite[Corollary 15]{DR2}. In fact, take a coresolution of fp-injective functors as in \eqref{eq. recolCOver1} for $F$, then $\vartheta(F):={\rm Ker}f$.

\s By \cite[Corollary 15]{DR2}, the functor $R_0$ is the zeroth right derived functor of $t$. Let $M$ be an $R$-module. Pick an injective copresentation $0 \rt M \rt I_0\rt I_1$, and define $R_0(M)$ as follows
\begin{equation*}
\begin{tikzcd}
0 \rar & R_0(M) \rar & (-\otimes I_0) \rar & (-\otimes I_1).
\end{tikzcd}
\end{equation*}

\s The left adjoint functor $j_{\la}:(\mmod R, \mathcal{A}{\rm b})\rt (\underline{{\rm mod}}\mbox{-}R, \mathcal{A}{\rm b})$ of $j$ is defined as follows. Let $F$ be a functor with an fp-injective copresentation as in \eqref{eq. recolCOver1}. We have the commutative diagram
\[ \xymatrix@C-0.5pc@R-0.5pc{&&0\ar[d]&  &&\\ &&F^1\ar[d]&  && \\  &  & -\otimes {\rm Ker}f  \ar@{.>}[dl]\ar[d] &  & &  \\
0\ar[r] & F\ar[r]\ar[d] & -\otimes E^0 \ar[r]^{-\otimes f}\ar[d] & -\otimes E^1\ar@{=}[d] \ar[r]^{-\otimes g} & -\otimes L\ar[r] \ar@{=}[d] & 0 \\
0 \ar[r] &F^0\ar[r]\ar[d]&-\otimes {\rm Im}f \ar[r]\ar[d]& -\otimes E^1 \ar[r]^{-\otimes g}\ar[d] &-\otimes L\ar[r]\ar[d] &0 \\ &0&0& 0 &0& }\]
Then $j_{\la}(F)= F^1$.

\s \label{Functor J rho} Finally to define $j_{\rho}$, let $F$ be an object of $(\mmod R, \mathcal{A}{\rm b})$ with an fp-injective copresentation as in \eqref{eq. recolCOver1}. We can construct the commutative
\[\xymatrix{0 \ar[r]& {\rm Ker} f \ar[r] \ar@{=}[d] & E^0 \ar[r]^f \ar[d]^r & E^1 \ar[d]^s \\ 0 \ar[r]& {\rm Ker}f \ar[r] & I^0 \ar[r] & I^1  }\]
with exact rows, where the lower row is a minimal injective copresentation of ${\Ker}f.$ By applying the tensor functor to the right square of the above diagram we get the diagram
\[\xymatrix{0 \ar[r]& F \ar[r] \ar[d]^{\eta_F} & -\otimes E^0 \ar[r]^{-\otimes f} \ar[d]^{-\otimes r} & -\otimes E^1  \ar[d]^{-\otimes s}    \\	0 \ar[r]& R_0(\vartheta(F)) \ar[r] & -\otimes I^0 \ar[r] & -\otimes I^1. }\]

Hence we have the unit $\eta_F.$ In view of \cite[Lemma 5]{DR2}, $j_{\rho}(F):={\rm Ker} \eta_F$, with the injective copresentation
\begin{equation*}
\begin{tikzcd}
0 \rar &  j_{\rho}(F) \rar & -\otimes E^0 \rar{-\otimes\left[\begin{smallmatrix}
f\\ r \end{smallmatrix}\right]  } &  -\otimes (E^1\oplus I_0).
\end{tikzcd}
\end{equation*}

\section{Morphism category of pure-injectives}
The morphism category $\CH(R\text{-}{\rm Mod})$ of $R\text{-}{\rm Mod}$ is a category whose objects are morphisms $f : X \rt Y$ in $R\text{-}{\rm Mod}$. If we regard the morphism $f : X \rt Y$ as an object in $\CH(R\mbox{-}{\rm Mod})$, we will denote it by $(X \st{f}\rt Y)$. A morphism from $(X\st{f}\rt Y)$ to $(X'\st{f'}\rt Y')$ is a pair $(\sigma_1, \sigma_2)$,  where $\sigma_1: X\rt X'$ and $\sigma_2:Y\rt Y'$ are morphisms in $R\text{-}{\rm Mod}$ and $\sigma_2f=f'\sigma_1.$
It is well-known that $\CH(R\mbox{-}{\rm Mod})$ is an abelian category with enough projective and enough injective objects. Two important subcategories of it are
\[\begin{array}{lll}
\CS(R\mbox{-}{\rm Mod}) & = & {\rm All \ monomorphisms \ in} \ \CH(R\mbox{-}{\rm Mod}), \ {\rm and}\\
\CF(R\mbox{-}{\rm Mod}) & = & {\rm All \ epimorphisms \ in} \ \CH(R\mbox{-}{\rm Mod}).
\end{array}\]

We also define the following two subcategories of $\CH(R)$.
\[\begin{array}{lll}
\CS(R\text{-}{\rm Pinj}) & = & {\rm the \ subcategory \ of} \ \CS(R\mbox{-}{\rm Mod}) \ {\rm consisting \ of \ all \ monomorphisms} \  (X \xrightarrow{f} Y) \\ && {\rm such \ that} \ X \ {\rm and}  \ Y  {\rm are \ pure\mbox{-}injectives},\\
\CF(R\text{-}{\rm Pinj}) & = & {\rm the \ subcategory \ of} \ \CF(R\mbox{-}{\rm Mod}) \ {\rm consisting \ of \ all \ epimorphisms} \  (Y \xrightarrow{g} Z) \ {\rm such} \\ && {\rm that} \ \Ker g \ {\rm and}  \ Y  {\rm are \ pure\mbox{-}injectives}.
\end{array}\]

 In this section, we examine the connections between the morphism category and its aforementioned subcategories with the categories of contravariant, resp. covariant, functors on finitely presented modules. To this end, we need to recall the notion of objective functors.

 \s \label{ObjectiveFunctor} Let $F:\CX \to \CY$ be an additive functor between additive categories.  An object $X$ in $\CX$ is called a kernel object for $E$ provided $F(X)=0$.  The  functor $F$ is called an objective functor if any morphism $h$ in $\CX$ with $F(h)=0$ factors through a kernel object of $F$. We say that the kernel of an objective functor $E$ is generated by a subcategory $\DD$ of $\CX$ if ${\rm add}\text{-}\DD$ coincides with the subcategory of all kernel objects of $F$. According to  \cite[Appendix]{RZ}, if $F$ is full, dense and objective, and if the kernel of $F$ is generated by a subcategory $\mathcal{D}$ of $\CX$, then $F$ induces an equivalence $\overline{F}:{\CX}/{\DD} \rt \CY$. Here, ${\CX}/{\DD}$ is the factor category of $\CX$ by the ideal of morphisms that factor through a finite direct sum of objects in $\DD.$

\subsection{The functor $\Psi$}\label{SubPhi}
We define a functor $\Psi:\CS(R\mbox{-}{\rm Mod})\rt ((R\text{-}\underline{{\rm mod}})^{\rm op}, \mathcal{A}{\rm b})$ as follows. Let $(X\st{f}\rt Y)$ be an object of  $\CS(R\mbox{-}{\rm Mod})$.
It induces the exact sequence
\begin{equation*}
\begin{tikzcd}
0 \rar & X \rar{f} & Y \rar{g} & Z \rar & 0,
\end{tikzcd}
\end{equation*}
where $Z:= {\rm Cok} f$ and $g$ is the canonical epimorphism. By Yoneda Lemma we have the exact sequence
\begin{equation}
\begin{tikzcd}
0 \rar & (-, X) \rar{{(-,f)}} & (-,Y) \rar{(-,g)} & (-,Z) \rar & F \rar & 0,
\end{tikzcd}
\end{equation}
where $F:= {\rm Cok}(-,g)$. We define
\[ \Psi(X\st{f}\rt Y):= F.\]
Note that since $F(P)=0$ for all projective objects $P$, we can view $F$ as an object in $((R\text{-}\underline{{\rm mod}})^{\rm op}, \mathcal{A}{\rm b})$.

Let $h =(h_1, h_2) \colon (X\st{f}\rt Y) \rt (X'\st{f'}\rt Y')$ be a morphism in $S(R\mbox{-}{\rm Mod})$.
It induces the commutative diagram
\begin{equation}\label{eq. com. the first functor}
\xymatrix{
0 \ar[r] & X \ar[d]^{h_1} \ar[r]^f & Y \ar[d]^{h_2} \ar[r]^g & Z \ar[d]^{h_3} \ar[r] & 0 \\
0 \ar[r] & X' \ar[r]^{f'} & Y' \ar[r]^{g'} & Z' \ar[r] & 0}
\end{equation}
\noindent
with exact rows, where $Z:= {\rm Cok}f,\, Z':= {\rm Cok}g$; and $h_3$ is the unique morphism making this diagram commutative. Define $\Psi(h)$  to be the unique morphism that makes the diagram
\[
\xymatrix{
0 \ar[r] & (-, X) \ar[d]^{(-, h_1)} \ar[r]^{(-, f)} & (-, Y) \ar[d]^{(-, h_2)} \ar[r] & (-, Z) \ar[d]^{(-, h_3)} \ar[r] & \Psi(f)\ar[r]\ar[d]^{\Psi(h)} & 0\\
0 \ar[r] & (-, X') \ar[r]^{(-, f')} & (-, Y') \ar[r] & (-, Z') \ar[r] & \Psi(f')\ar[r]& 0
}
\]
commutative.

\begin{theorem}\label{Thm.Psi}
The restriction of the functor $\Psi$ to $\CS(R\mbox{-}{\rm Pinj})$, i.e.
\[\Psi|:\CS(R\mbox{-}{\rm Pinj})\rt ((R\text{-}\underline{{\rm mod}})^{\rm op}, \mathcal{A}{\rm b}),\]
is full, dense and  objective.
\end{theorem}

\begin{proof}
Let $\delta: F\rt F'$ be a morphism in $((R\text{-}\underline{{\rm mod}})^{\rm op}, \mathcal{A}{\rm b})$. Consider special flat resolutions of $F$ and $F'$ as in the Construction \ref{ConsFlatRes},
\begin{equation}\label{Dig.Proof.full}
\xymatrix{0 \ar[r] & (-, E_1) \ar@{.>}[d]^{\eta_1=(-, \sigma_1)} \ar[r]^{(-, d_1)} & (-, E_0) \ar@{.>}[d]^{\eta_0=(-, \sigma_0)} \ar[r]^{(-, d_0)} & (-, L) \ar@{.>}[d]^{\eta=(-, \sigma)} \ar[r]^{\pi} & F \ar[d]^{\delta} \ar[r] & 0 \\ 0  \ar[r] & (-, E_1')\ar[r]^{(-, d_1')} & (-, E_0')\ar[r]^{(-, d'_0)} & (-,L' )\ar[r]^{\pi'} & F' \ar[r] & 0.}
\end{equation}
Based on the properties of the special flat resolutions, the morphism $\delta$ can be lifted to the flat resolutions. The embedding $\Upsilon$, implies that the morphisms $\eta, \eta_0$ and $\eta_1$ are presented by the morphisms $\sigma, \sigma_0$ and $\sigma_1,$ respectively. Hence, $\Psi(\sigma_1, \sigma_0)=\delta.$ Consequently,  $\Psi|$ is full.

Let $F$ be an object in $((R\text{-}\underline{{\rm mod}})^{\rm op}, \mathcal{A}{\rm b})$. We can identify $F$ as an object in $((R\text{-}{\rm mod})^{\rm op}, \mathcal{A}{\rm b})$. Consider a special flat resolution
\begin{equation}\label{eq. flat.reso.}
\begin{tikzcd}
0 \rar & ( - , E_1) \rar{{(-,d_1)}} & ( - , E_0) \rar{(-,d_0)} & ( - , L) \rar{\pi} & F \rar & 0,
\end{tikzcd}
\end{equation}
of $F$, where $E_0, E_1$ are pure-injective $R$-modules. By evaluating on the right module $R$ and noting that $F(R)=0$,  we obtain the monomorphism $d_1:E_1 \lrt E_0$ with ${\rm Cok}d_1\simeq L$. Hence, the object $(E_1\st{d_1}\lrt E_0)$ lies in $\CS(R\text{-}{\rm Pinj})$ with $\Psi(d_1)=F$. So $\Psi|$ is dense.

Finally, let $(\sigma_1, \sigma_0):(E_1\st{d_1}\rt E_0)\rt (E_1'\st{d_1'}\rt E_0)$ be a morphism in $\CS(R\mbox{-}{\rm Pinj})$ such that $\Psi(\sigma_1, \sigma_0)=0$. Keep the notations as in the Diagram \eqref{Dig.Proof.full}.  Since $\delta=\Psi(\sigma_1, \sigma_0)=0$, there exists $\delta_1=(-, L)\rt (-, E_0')$ such that $(-, d_0')\delta_1=\eta$. In a similar way, we can see that there exists $\delta_2:(-, E_0)\rt ( - , E_1')$ such that $\eta_0=(-, d_1')\delta_2+\delta_1( - , d_0)$. Since $(-, d_1')$ is injective, one can see that $\delta_2  (-, d_1)=\eta_1$. Moreover, $\delta_2=( - , s_2)$ and $\delta_1=(-, s_1)$ can be presented by morphisms $s_2: E_0\rt E_1'$ and $s_1: L\rt E_0'$, respectively. By applying the embedding $\Upsilon$ to transfer these equalities to $R\mbox{-}{\rm Mod}$, we get the commutative diagram
\[\xymatrix{    & &  E_1   \ar@/^1.25pc/@{.>}[dd]^<<{\sigma_1} \ar[d]_{\sigma_1}  \ar[r]^{ d_1 } & E_0 \ar@/^1pc/@{.>}[dd]^<<<<<{\sigma_0} \ar[d]_{\left[\begin{smallmatrix}
s_2\\ s_1d_0
\end{smallmatrix}\right] }  \ar[r]^{d_0}  &  L  \ar@/^1.2pc/@{.>}[dd]^{\sigma} \ar[d]_{s_1}   &  &   \\
		& & E_1' \ar[r]^{\left[\begin{smallmatrix}
1\\ 0
\end{smallmatrix}\right] } \ar@{=}[d] & E_1'\oplus E_0'  \ar[r]^{[0~~1]} \ar[d]_{[d_1'~~1]}  &   E_0' \ar[d]_{d_0'}    &
		  \\
		& & E_1'  \ar[r]^{d_1'}  &  E_0' \ar[r]^{ d_0'}  &L'  &  }\]

In particular, the morphism $(\sigma_1, \sigma_0)$ factors through the kernel object $(E_1'\st{1}\rt E_1')\oplus (0 \rt E_0')$. Hence $\Psi|$ is objective.
\end{proof}

In view of \ref{ObjectiveFunctor}, we have the following corollary.

\begin{corollary}\label{Cor. Psi}
With the notations used in Theorem \ref{Thm.Psi}, the functor $\Psi|$ induces an equivalence $\Psi'$ of categories
\[\Psi':\CS(R\mbox{-}{\rm Pinj})/\CV\st{\sim}\lrt  ((R\text{-}\underline{{\rm mod}})^{\rm op}, \mathcal{A}{\rm b}),\]
where $\CV$ denotes the smallest additive subcategory of $\CS(R\mbox{-}{\rm Pinj})$ containing all objects of the form $(E\st{1}\rt E)$ and $(0\rt E)$, where $E$ is a pure-injective module.
\end{corollary}

\subsection{The functor $\Phi$}
In this subsection, we define a functor $\Phi:\CF(R\mbox{-}{\rm Pinj}) \rt (\underline{{\rm mod}}\mbox{-}R, \CA{\rm b})$ as follows. Let $(Y\st{g}\rt Z)$ be an object of $\CF(R\mbox{-}{\rm Pinj}).$ It induces the exact sequence
\begin{equation*}
\begin{tikzcd}
0 \rar & X \rar{f} & Y \rar{g} & Z \rar & 0,
\end{tikzcd}
\end{equation*}
where $X:= \Ker g$ and $f$ is the inclusion map. This, in turn, induces the exact sequence
\begin{equation}
\begin{tikzcd}
0 \rar & G \rar & -\otimes X \rar{-\otimes f } &  -\otimes Y \rar{-\otimes g} & -\otimes Z \rar & 0,
\end{tikzcd}
\end{equation}
where $G:= \Ker (-\otimes f)$. Set \[\Phi(Y\st{g}\lrt Z):= G.\]

Note that since $G(P)=0$, for all projective objects, we can view $G$ as an object in $(\underline{{\rm mod}}\mbox{-}R, \CA{\rm b})$.
Let $h =(h_2, h_3) \colon (Y\st{f}\rt Z) \rt (Y'\st{g}\rt Z')$ be a morphism in $\CF(R \mbox{-}{\rm Pinj})$. It induces a commutative diagram as in \eqref{eq. com. the first functor}. Define $\Phi(h)$  to be the unique morphism that makes the following diagram commutative
\[
\xymatrix{
0 \ar[r] & \Phi({g}) \ar[d]^{\Phi(h)} \ar[r] & -\otimes X \ar[d]^{-\otimes h_1}\ar[r]^{-\otimes f} & -\otimes Y \ar[d]^{-\otimes h_2} \ar[r]^{-\otimes g} & -\otimes Z\ar[r]\ar[d]^{-\otimes h_3} & 0\\
0 \ar[r] & \Phi({g'}) \ar[r] & -\otimes X' \ar[r]^{-\otimes f'} & -\otimes Y' \ar[r]^{-\otimes g'} & -\otimes Z'\ar[r]& 0.
}
\]

\begin{theorem}\label{Thm.Phi}
 The functor $\Phi:\CF(R\mbox{-}{\rm Pinj}) \rt (\underline{{\rm mod}}\mbox{-}R, \CA{\rm b})$  is full, dense and  objective.
\end{theorem}

\begin{proof}
The proof follows a dual approach compared to the argument presented in Theorem \ref{Thm.Psi}. In this case, we utilize injective coresolutions instead of the special flat resolutions.
\end{proof}

\begin{corollary}\label{Cor.Phi}
With the notations and assumptions of Theorem \ref{Thm.Phi}, the functor $\Phi$ induces an equivalence of categories
\[\Phi':\CF(R\mbox{-}{\rm Pinj})/\CU\st{\sim}\rt  (\underline{{\rm mod}}\mbox{-}R, \CA{\rm b}),\]
where $\CU$ denotes the smallest additive subcategory of $\CF(R\mbox{-}{\rm Pinj})$ consisting of all objects of the form $(E\st{1}\rt E)$ and $(E \rt 0)$.
\end{corollary}

Recall that a ring $R$ is called left pure-semisimple if every left $R$-module is pure-injective.

\begin{corollary}\label{Cororpuresimple}
Let $R$ be a left pure-semisimple ring. Then, there exists the following equivalences.
\begin{itemize}
\item [$(1)$] \[\CS(R\mbox{-}{\rm Mod})/\CV\st{\sim}\rt  ((R\text{-}\underline{{\rm mod}})^{\rm op}, \mathcal{A}{\rm b}),\]
where $\CV$ denotes the smallest additive subcategory of $\CS(R\mbox{-}{\rm Mod})$ containing all objects of the forms $(M\st{1}\rt M)$ and $(0\rt M)$, where $M$ ranges over all modules.
\item [$(2)$] \[\CF(R\mbox{-}{\rm Mod})/\CU\st{\sim}\rt  (\underline{{\rm mod}}\mbox{-}R, \CA{\rm b}),\]
where $\CU$ denotes the smallest additive subcategory of $\CF(R\mbox{-}{\rm Mod})$ containing all objects of the form $(N\st{1}\rt N)$ and $(N \rt 0)$, where $N$ ranges over all modules.
\end{itemize}
\end{corollary}

Let $\La$ be an artin algebra of finite representation type. Let $M$ be a basic module such that $\La\text{-}\underline{{\rm mod}}={\rm add}\text{-}M$. The stable Auslander algebra $\Gamma$ of $\La$ is the opposite endomorphism algebra of $M$ in the stable category $\La\text{-}\underline{{\rm mod}}$, i.e., $\Gamma=\underline{{\rm End}}(M)^{\rm op}$.

\begin{corollary}\label{CorofinitRep}
Let $\La$ be an artin algebra of finite representation type and let $\Gamma $ be the stable Auslander algebra of $\La.$ Then, we have the following equivalences.
\begin{itemize}
\item [$(1)$] \[\CS(\La\mbox{-}{\rm Mod})/\CV\st{\sim}\rt \Gamma \mbox{-}{\rm Mod},\]
where $\CV$ denotes the smallest additive subcategory of $\CS(\La\mbox{-}{\rm Mod})$ consisting of all objects of the form $(M\st{1}\rt M)$ and $ (0\rt M)$, where $M$ ranges over all modules.
\item [$(2)$] \[\CF(\La\mbox{-}{\rm Mod})/\CU\st{\sim}\rt \Gamma \mbox{-}{\rm Mod},\]
where $\CU$ denotes the smallest additive subcategory of $\CF(\La\mbox{-}{\rm Mod})$ consisting of all objects of the form $(N\st{1}\rt N)$ and $(N\rt 0)$, where $N$ ranges over all modules.
\end{itemize}
\end{corollary}

\begin{proof}
 Let $M$ be a basic module such that $\La\text{-}\underline{{\rm mod}}={\rm add}\text{-}M$.

 $(1)$ By evaluating on the  module $M$, i.e. $F\mapsto F(M)$,  we get the equivalence
\[((\La\text{-}\underline{{\rm mod}})^{\rm op}, \mathcal{A}{\rm b})\simeq \Gamma \mbox{-}{\rm Mod}. \]
Next, by combing the above equivalence and the one in Corollary \ref{Cororpuresimple}, we obtain the desired equivalence.

$(2)$ Since $M$ is an additive generator of $\La\text{-}\underline{{\rm mod}}$, we infer that $D(M)$ is an additive generator for $\underline{{\rm mod}}\mbox{-}\La$. Similar to $(1)$, by evaluating on $D(M)$, we get the equivalence $(\underline{{\rm mod}}\mbox{-}\La, \CA{\rm b})\simeq \Gamma \mbox{-}{\rm Mod}.$ Now the result follows from Corollary \ref{Cororpuresimple}.
\end{proof}

\begin{remark}
With the notations and assumptions of Corollary \ref{CorofinitRep}, we have the following  auto-equivalence on $\Gamma\mbox{-}{\rm Mod}$
\begin{equation*}
\begin{tikzcd}
\Gamma\mbox{-}{\rm Mod} \rar{\sim} & \CS(R\mbox{-}{\rm Mod})/\CV \rar{\sim} & \CF(R\mbox{-}{\rm Mod})/\CU  \rar{\sim} & \Gamma\mbox{-}{\rm Mod}.
\end{tikzcd}
\end{equation*}
\end{remark}

\s\label{subflatdimOne}
An object  $F \in ((R\text{-}{\rm mod})^{\rm op}, \mathcal{A}{\rm b})$ is of flat dimension at most one if it admits a flat resolution
\begin{equation*}
\begin{tikzcd}
0 \rar & (-, M) \rar & (-, N) \rar & F \rar & 0.
\end{tikzcd}
\end{equation*}

Based on Construction \ref{ConsFlatRes}, we can infer that $F$ is of flat dimension at most one if and only if $M$ in a flat resolution as above is pure-injective.

\begin{notation}
 Let $\CS_0(R\mbox{-}{\rm Pinj})$ represent the subcategory of $\CS(R\mbox{-}{\rm Mod})$ consisting of all monomorphisms $(M\st{f}\rt N)$, where $M$ is a pure-injective module. Additionally, let $\CF^{\leq 1}$ denote the subcategory of all functors that have a flat dimension of at most one.
\end{notation}

\s \label{Functor Gamma} Consider the functor
\[\begin{array}{lll}
\Sigma: \CH(R\mbox{-}{\rm Mod}) & \lrt & ((R\text{-}{\rm mod})^{\rm op}, \mathcal{A}{\rm b}) \\
 \ \ \ \ \  (X\st{f}\rt Y) & \  \mapsto & {\rm Cok}(\Upsilon(X)\st{\Upsilon(f)}\lrt \Upsilon(Y))
\end{array}\]
where $\Upsilon$ is the functor defined in \eqref{Upsilonfunc}.\\

By using a similar argument as in the proof of Theorem \ref{Thm.Psi}, we can establish the following proposition; thus, we will skip the proof.

\begin{proposition}\label{flatdimension}
The functor $\Sigma |:\CS_0(R\mbox{-}{\rm Pinj}) \rt \CF^{\leq 1}$ is full, dense and objective. In particular, it induces an equivalence
\[\CS_0(R\mbox{-}{\rm Pinj})/\CK\simeq \CF^{\leq 1} \]
 where $\CK$ is the full subcategory generated by all isomorphisms.
\end{proposition}

\begin{notation}
Let $\CF_0(R\mbox{-}{\rm Pinj})$ denote the subcategory of  $\CF(R\mbox{-}{\rm Mod})$   consisting of all epimophisms $(M\st{f}\rt N)$ such that $M$ is pure-injective. Moreover, let ${\rm fp}^{\leq 1}$ denote the subcategory of $({\rm mod}\mbox{-}R, \CA{\rm b})$ consisting of all functors  $F$ which admits an fp-injective coresolution
\begin{equation*}
\begin{tikzcd}
0 \rar & F\rar &  -\otimes M \rar & -\otimes N \rar & 0,
\end{tikzcd}
\end{equation*}
where $M$ is a pure-injective module.
\end{notation}

\s \label{Functor Gamma prime} Consider the functor
\[\begin{array}{lll}
\Sigma': \CH(R\mbox{-}{\rm Mod}) \lrt  (R\text{-}{\rm mod}, \mathcal{A}{\rm b}) \vspace{0.1cm}\\
 \ \ \ \ \ \ \  (X\st{f}\rt Y) \  \mapsto  {\rm Ker}(t( X) \st{t(f)}\lrt  t(Y))
\end{array}\]
where the functor $t$ is defined in \eqref{functor t}.

We can also prove the following result. However, because of the similarity in the arguments, we will skip the proof, see Theorem \ref{Thm.Phi}.

\begin{proposition}\label{fp-inj-dimension}
The functor $\Sigma' |:\CF_0(R\mbox{-}{\rm Pinj}) \rt {\rm fp}^{\leq 1}$ is full, dense and objective. In particular, there exists an equivalence \[\CF_0(R\mbox{-}{\rm Pinj})/\CL\simeq {\rm fp}^{\leq 1} \]
where $\CL$ is the full subcategory generated by all isomorphisms.
\end{proposition}

\section{Connections to Auslander-Bridger Transpose}
The categories $\CH(R\mbox{-}{\rm Mod})$, $\CS(R\mbox{-}{\rm Mod})$ and $\CF(R\mbox{-}{\rm Mod})$ are related by the cokernel and kernel functors as follows
\begin{align*}
{\rm Cok}: & \CH(R\mbox{-}{\rm Mod})\lrt \CF(R\mbox{-}{\rm Mod}), \quad (A \st{f}\lrt  B)
   \mapsto (B\st{\text{can}}\lrt {\rm Cok} f),\\
{\rm Ker}: & \CH(R\mbox{-}{\rm Mod})\lrt \CS(R\mbox{-}{\rm Mod}),  \quad (B \st{g}\lrt C)
   \mapsto({\rm Ker} g \st{\text{inc}} \lrt B).
\end{align*}
The restrictions of these functors to $\CS(R\mbox{-}{\rm Pinj})$ and $\CF(R\mbox{-}{\rm Pinj})$ induce the pair
\[\begin{tikzcd}[column sep = large]
\CS(R\mbox{-}{\rm Pinj}) \rar["{\rm Cok}", shift left]  & \CF(R\mbox{-}{\rm Pinj}) \lar["{{\rm Ker}}", shift left]
\end{tikzcd}     \]
of inverse equivalences.

\s Since ${\rm Cok}(\CV)=\CU$ and ${\rm Ker}(\CU)=\CV$,  by combing Corollaries \ref{Cor. Psi} and \ref{Cor.Phi},  we obtain the following inverse equivalences, compare \cite[Page 382]{He}.
\[\begin{tikzcd}[column sep = large]
\Xi:  \  ((R\text{-}\underline{{\rm mod}})^{\rm op}, \mathcal{A}{\rm b}) \rar["\Psi'^{-1}", shift left] & \CS(R\mbox{-}{\rm Pinj})/\CV \lar["{{\Psi'}}", shift left] \rar["{\rm Cok}", shift left]  & \CF(R\mbox{-}{\rm Pinj})/\CU \lar["{{\rm Ker}}", shift left] \rar["\Phi'", shift left]&  (\underline{{\rm mod}}\mbox{-}R, \CA{\rm b}).\lar["{\Phi'^{-1}}", shift left] &&
\end{tikzcd} \]

In view of the Yoneda Lemma, we have equivalences
\[{\rm prj}\mbox{-}((R\text{-}\underline{{\rm mod}})^{\rm op}, \mathcal{A}{\rm b})\simeq R\text{-}\underline{{\rm mod}} \ {\rm and} \ {\rm prj}\mbox{-}(\underline{{\rm mod}}\mbox{-}R, \CA{\rm b})\simeq (\underline{{\rm mod}}\mbox{-}R)^{\rm op}.\]
If we restrict the  equivalence $\Xi: \Phi' \circ {\rm Cok}\circ \Psi'^{-1}$ to finitely generated projective functors, we reach the duality $\varrho$ between $R\text{-}\underline{{\rm mod}}$ and $\underline{{\rm mod}}\mbox{-}R$
\[\xymatrix{
  ((R\text{-}\underline{{\rm mod}})^{\rm op}, \mathcal{A}{\rm b} ) \ar[rr]^{\Phi' \circ {\rm Cok}\circ \Psi'^{-1}} & & (\underline{{\rm mod}}\mbox{-}R, \CA{\rm b})  & \\
  R\text{-}\underline{{\rm mod}}\ar[rr]^{\varrho}\ar[u]& &  \underline{{\rm mod}}\mbox{-}R.\ar[u] &
}\]

The duality $\varrho$ appearing in the lower row of the above diagram, provides motivation for this section; we plan to discuss the relations between the dualities ${\rm Tr}$ and $\varrho$.

\s Recall that the Auslander-Bridger Transpose, Tr, provides a duality of categories
\[{\rm Tr}:R\mbox{-}{\underline{\rm mod}}\rt \underline{{\rm mod}}\mbox{-}R,\]
which is defined as follows. Let $M$ be a finitely presented left $R$-module with a presentation
\begin{equation*}
\begin{tikzcd}
R^m \rar{f} & R^n \rar & M \rar & 0,
\end{tikzcd}
\end{equation*}
where $m, n$ are finite. Applying the left exact contravariant functor $( - )^*=( - , R)$ induces an exact sequence of left $R$-modules,
\begin{equation*}
\begin{tikzcd}
0 \rar & M^* \rar & (R^n)^* \rar{f^*} & (R^m)^* \rar & {\rm Tr}(M) \rar & 0
\end{tikzcd}
\end{equation*}
of right $R$-modules, where ${\rm Tr}(M)$ is defined to be the cokernel of $f^*$.

\begin{notation}
For  modules $M$ and $N$  in $R\mbox{-}{\rm mod}$ and $\mmod R$, we denote by $(-, \underline{M})$ and $(\underline{M}, -)$ the associated representable functors in $((R\text{-}\underline{{\rm mod}})^{\rm op}, \mathcal{A}{\rm b})$ and $(\underline{{\rm mod}}\mbox{-}R, \CA{\rm b})$, respectively.
\end{notation}

In view of \cite[Theorem 3.5.1]{Xu}, a ring $R$ is called left Xu if  the class of pure-injective left $R$-modules is closed under extensions. This is equivalent to the condition that for every left $R$-module M, the cokernel $PE(M)/M$ of the pure-injective envelope of $M$ is flat. Throughout, for simplicity, we denote $\Tor_1^R$ by $\Tor$.

\begin{proposition}
Assume that $R$ is a left Xu ring. Then, for  every left finitely presented module $Z$, there exists an isomorphism $(\underline{\varrho(Z)}, -) \simeq \Tor(-, Z).$
\end{proposition}

\begin{proof}
Let $Z$ be a finitely presented module. Consider the flat resolution
\begin{equation*}
  \begin{tikzcd}
0 \rar & (-, \Omega_{\La}(Z)) \rar & ( - , P) \rar & ( - , Z) \rar & ( - , \underline{Z}) \rar & 0
  \end{tikzcd}
\end{equation*}
of the functor $(-, \underline{Z})$ in $((R\text{-}\underline{{\rm mod}})^{\rm op}, \mathcal{A}{\rm b})$, where $P$ is a projective module. We use diagrams in  Construction \ref{ConsFlatRes}.  From diagram \eqref{digarm1Con} we obtain the short exact sequence
\begin{equation*}
\begin{tikzcd}
0 \rar & P \rar & N \rar & PE(\Omega_{\La}(Z))/\Omega_{\La}(Z) \rar & 0.
\end{tikzcd}
\end{equation*}
Furthermore, consider the short exact sequence
\begin{equation*}
\begin{tikzcd}
0 \rar & N \rar & PE(N) \rar & PE(N)/N \rar & 0.
\end{tikzcd}
\end{equation*}
We know that $P$ is a projective module and that both $PE(\Omega_{\La}(Z))/\Omega_{\La}(Z)$ and $PE(N)/ N$ are flat since $R$ is Xu. According to the  long exact sequence associated with these two short exact sequences,  we infer that $\Tor_i^R( - , PE(N))=0$, for all $i \geq 1$. Consider the special flat resolution of $F$ as described in \eqref{SequCon2}. Since $\Tor( - , PE(N))=0$, we get the following sequences of isomorphisms
\begin{align*}
(\underline{\varrho(Z)}, -)  &=\Xi(-, \underline{Z})\\
      &= \Phi \circ {\rm Cok} \circ \Psi^{-1}(-, \underline{Z})\\
      &=\Phi \circ {\rm Cok} (PE(\Omega_{\La}(Z)) \st{h}\rt PE(N) )\\
      &=\Phi(PE(\Omega_{\La}(Z)) \st{g}\rt L)\\
      &=\Tor(-, L)
 \end{align*}
By diagram \eqref{diagram3Con} in the construction, we have the short exact sequence
\begin{equation*}
\begin{tikzcd}
0 \rar & Z \rar & L \rar & PE(N)/N \rar & 0.
\end{tikzcd}
\end{equation*}
Since $PE(\Omega_{\La}(Z))/ \Omega_{\La}(Z)$ is flat, the long exact sequence associated with this short exact sequence shows that $\Tor(-, Z) \simeq \Tor(-, L)$. This completes the proof.
\end{proof}

By \cite[Section 3]{He}, the Auslander-Bridger transpose induces an equivalence of functor categories
\[{\rm Tr}^*:((R\text{-}\underline{{\rm mod}})^{\rm op}, \mathcal{A}{\rm b}) \rt (\underline{{\rm mod}}\mbox{-}R, \CA{\rm b}),\]
given by ${\rm Tr}^*:F \mapsto F \circ {\rm Tr}$.

\begin{proposition}\label{ConnTranspose}
Assume that $R$ is an artinian ring with Morita self-duality; for instance, $R$ can be an artin algebra. Then, for every left finitely generated module $M$, there is an isomorphism $\varrho(M)\simeq {\rm Tr}(M)$ in the stable category $\underline{{\rm mod}}\mbox{-}R$. In particular, there exists an isomorphism $\Xi \simeq {\rm Tr}^*.$
\end{proposition}

\begin{proof}
Let $M$ be a left finitely generated module. Consider the flat resolution
\begin{equation}\label{eq.Prop. Morita-self}
\begin{tikzcd}
0 \rar & (-, \Omega_{\La}(M))\rar{(-, i)} & ( - , P)\rar{(-, f)} & ( - , M) \rar & ( - , \underline{M}) \rar & 0,
\end{tikzcd}
\end{equation}
of $(-, \underline{M})$, which is induced by the short exact sequence
\begin{equation*}
\begin{tikzcd}
0 \rar & \Omega_{\La}(M) \rar{i} & P \rar{f} & M \rar & 0,
\end{tikzcd}
\end{equation*}
where $P$ is a projective module. Since all the terms in the latter  exact sequence are finitely generated, the sequence \eqref{eq.Prop. Morita-self} provides a special flat resolution of $(-, \underline{M})$. Hence, by the definition of $\Xi$ we have the following equalities
  \begin{align*}
      (\underline{\varrho(M)}, -)  &=\Xi(-, \underline{M})\\
      &= \Phi \circ {\rm Cok} \circ \Psi^{-1}(-, \underline{M})\\
      &=\Phi \circ {\rm Cok} (\Omega_{\La}(M) \st{i}\rt P )\\
      &=\Phi(P\st{f}\rt M)\\
      &=\Tor(-, M).
    \end{align*}
It is known that $\Tor(-, M)\simeq (\underline{{\rm Tr}M}, -)$. Therefore, we have the isomorphism $(\underline{\varrho(M)}, -)\simeq (\underline{{\rm Tr}M}, -)$, which implies that  ${{\varrho(M)}}\simeq {{\rm Tr}M}$ in the stable category $\underline{{\rm mod}}\mbox{-}R$.
\end{proof}

\section{Quasi-Frobenius rings}
Let $\CS(n)$ be the submodule category of $\mmod\La_n$, where $\La_n=k[x]/{\lan x^n \ran}$ and $k$ is a field. In \cite{RZ}, Ringel and Zhang studied two functors $F$ and $G$ from $\CS(n)$ to $\mmod \Pi_{n-1}$, where $\Pi_{n-1}$ denotes the preprojective algebra of type $\mathbb{A}_{n-1}$. The functor $R$ has been studied in \cite{LZ}. They showed that $\Pi_{n-1}$ is isomorphic to $\underline{\Ga}$, the stable Auslander algebra of $\La$. In Construction \ref{SubPhi}, we studied the large module version of $F$. In this section, we focus on the other functor and explore its extended version for large modules, as well as its covariant version.

\s Recall that a ring R is Quasi-Frobenius, QF for short, if it is left and right artinian and left and right self-injective \cite{NY}. If $R$ is QF, then it is left perfect, so every flat left $R$-module is projective. Furthermore, if $R$ is a QF ring, then a left $R$-module is projective if and only if it is injective. Therefore $R\text{-}\underline{{\rm Mod}} = R\text{-}\overline{{\rm Mod}}$. Moreover, over QF rings, every module $M$ has a projective cover $\pi : P \lrt M$, whose kernel is denoted by $\Omega(M)$. For more properties of QF rings see \cite[\S 4.1]{He}.

In this section, we assume that $R$ is a QF ring.

\subsection{Contravariant version}
Recall from \ref{Functor Gamma} the functor\\
\[\begin{array}{lll}
\Sigma:\CH(R\mbox{-}{\rm Mod}) & \lrt & ((R\text{-}{\rm mod})^{\rm op}, \mathcal{A}{\rm b}), \vspace{0.1cm} \\
 \ \ \ \ \  (X\st{f}\rt Y) & \  \mapsto & {\rm Cok}(\Upsilon(X)\st{\Upsilon(f)} \lrt \Upsilon(Y))
\end{array}\]
where the functor $\Upsilon$ is defined in \eqref{Upsilonfunc}.

\begin{definition}
We define the functor $\Theta:\CS(R\mbox{-}{\rm Mod}) \lrt ((R\text{-}\underline{{\rm mod}})^{\rm op}, \mathcal{A}{\rm b}) $ to be the composition
\begin{equation*}
\begin{tikzcd}
\CS(R\mbox{-}{\rm Mod}) \ar[r,hook] & \CH(R\mbox{-}{\rm Mod}) \rar{\Sigma} & ((R\text{-}{\rm mod})^{\rm op}, \mathcal{A}{\rm b}) \rar{i_{\la}} & ((R\text{-}\underline{{\rm mod}})^{\rm op}, \mathcal{A}{\rm b}),
\end{tikzcd}
\end{equation*}
where $i_{\la}$ is defined in \ref{Subsect-ila}.
\end{definition}

Let us investigate the functor $\Theta$ more explicitly.  Let $(X \st{f}\rt Y)$ be an object of  $\CS(R\mbox{-}{\rm Mod})$. It induces the short exact sequence
\begin{equation*}
\begin{tikzcd}
0 \rar & X \rar{f} & Y \rar{g}  & Z \rar & 0
\end{tikzcd}
\end{equation*}
where $Z:= {\rm Cok} f$.

Take an epimorphism $d: P\rt Z$, where $P$ is a projective module. By the projectivity of $P$, there is a morphism $d_1$ which makes  the following diagram commutative
\begin{equation}
\xymatrix{ & & P \ar[d]^{d_1} \ar[r]^{d} & Z \ar@{=}[d] \ar[r] & 0 \\
0 \ar[r] & X \ar[r]^{f} & Y \ar[r]^{g} & Z \ar[r] & 0.}
\end{equation}

This, in turn, induces the epimorphism $P\oplus X \st{[d_1~~f]}\lrt Y$. By applying the Yoneda functor to this epimorphism, we obtain the exact sequence
\begin{equation*}
\begin{tikzcd}
 (-, P\oplus X) \rar{[d_1~~f]} & (-,Y) \rar & F \rar & 0.
\end{tikzcd}
\end{equation*}
By definition
\[\Theta(X\st{f}\lrt Y)= F.\]

\begin{remark}
Note that since $F$ is presented by an epimorphism, $F$ is indeed an object in $((R\text{-}\underline{{\rm mod}})^{\rm op}, \mathcal{A}{\rm b})$. It is routine to check that this definition is independent of the choice of the projective module $P$ as well as the morphism $d_1$.
\end{remark}

The proof of the following theorem is obtained by modifying the proof of Theorem 3.5 of \cite{AHS} in our context. We use freely the functors introduced in the recollement \eqref{RecollContr} throughout the proof.

\begin{theorem}\label{Theta-Theorem}
Let $R$ be a Quasi-Frobenius ring of finite representation type. Then, the functor
\[\Theta:\CS(R\mbox{-}{\rm Mod})\rt((R\text{-}\underline{{\rm mod}})^{\rm op}, \mathcal{A}{\rm b})\]
is full, dense and objective.
\end{theorem}

\begin{proof}
Let $(X\st{f}\rt Y)$ be an object in $\CS(R\mbox{-}{\rm Mod})$. We have the flat resolution
\begin{equation}\label{eq. Theoem.theta}
\begin{tikzcd}
0 \rar & (-, X) \rar{(-, f)} & ( - ,Y) \rar & \Sigma(f) \rar & 0.
\end{tikzcd}
\end{equation}
Therefore, to show that $\Theta$ is full, it is enough to show that $i_{\la}$ is full on $\Theta(\CS(R\mbox{-}{\rm Mod}))=\CF^{\leq 1}$, which consists of functors with a special flat resolution of length at most one.

Let $F$ and $G$ be functors in $\CF^{\leq 1}$ and $\gamma: i_{\la}(F)\rt i_{\la}(G)$ be a morphism in $((R\text{-}\underline{{\rm mod}})^{\rm op}, \mathcal{A}{\rm b})$. We have the commutative diagram
\[\xymatrix{L(\nu(F)) \ar[r]^{ \ \ \ \epsilon_F}  & F \ar[r] \ar@{.>}[d]^{\delta} & i_{\la}(F) \ar[r] \ar[d]^{\gamma}  & 0 \\ L(\va(G)) \ar[r]^{ \ \ \ \epsilon_G} & G \ar[r] & i_{\la}(G)\ar[r] & 0,}\]
with exact rows and the counits $\epsilon_F$ and $\epsilon_G$. Let $K$ denotes the kernel of $G \rt i_{\la}(G)$. We show that $\Ext^1(F, K)=0,$ where $\Ext^1(F, K)$ denote the abelian group of equivalence classes of extension of $F$ by $K$ in the category $((R\text{-}{\rm mod})^{\rm op}, \mathcal{A}{\rm b})$.  To this end, apply the known isomorphism
\[\Ext^1_{}(F, K)\simeq \Hom_{\mathbb{K}}(\mathbf{P}_F, \mathbf{P}_K[1]),\]
where $\mathbb{K}$ denotes the homotopy category of complexes of functors of $((R\text{-}{\rm mod})^{\rm op}, \mathcal{A}{\rm b})$, and $\mathbf{P}_F$ and $\mathbf{P}_K$ denote deleted projective resolutions of $F$ and $K$, respectively. Since $F \in \CF^{\leq 1}$, there is a flat resolution of $F$
\begin{equation*}
\begin{tikzcd}
0 \rar & ( - , A) \rar{( - , f)} & ( - , B) \rar & F \rar & 0,
\end{tikzcd}
\end{equation*}
where $f:A \rt B$ is a monomorphism. But, since $R$ is of finite representation type, $A$ and $B$ are pure-projective modules. Hence, it is indeed a projective resolution of $F$ in $((R\text{-}{\rm mod})^{\rm op}, \mathcal{A}{\rm b})$. Moreover, by definition of $L$, there is a projective presentation  $( - , P)\rt ( - , Q)\rt L(\nu(G))\rt 0$ of $L(\nu(G))$ such that $P, Q \in R\mbox{-}{\rm Prj}$. By the epimorphism $L(\nu(G)) \rt K \rt 0$, we can find a deleted  projective resolution of $K$
\begin{equation*}
\begin{tikzcd}
\mathbf{P}_K: 0 \rar & ( - , M) \rar & ( - , N) \rar & ( - , Q)\rar & 0,
\end{tikzcd}
\end{equation*}
where $( - , Q)$ is its zero's term.
Now consider a chain map
\[\xymatrix{0 \ar[r]&0 \ar[r] \ar[d]&0\ar[r] \ar[d] & ( - , A) \ar[r] \ar[d]^{( - , g)} & (-, B) \ar[r] \ar[d]  & 0 \\ 	0\ar[r] & ( - , M) \ar[r]&	( - , N) \ar[r] & ( - , Q) \ar[r] & 0 \ar[r] & 0}\]
from $\mathbf{P}_F$ to $\mathbf{P}_K[1]$. Since $Q$ is a projective module and $\La$ is a Quasi-Frobenius ring, $Q$ is injective and hence, in view of the embedding $\Upsilon$, the morphism $g: A \rt Q$ can be extended to a morphism $h: B \rt Q$. Applying the embedding $\Upsilon$, it implies the extension of the morphism $( - , g)$ to a morphism from $( - , B)$ to $( - , Q)$. This, in turn, implies that the chain map is null-homotopic. Hence $\Hom_{\mathbb{K}}(\mathbf{P}_F, \mathbf{P}_K[1])=0$. Therefore $\Ext^1(F, K)=0$. This implies that there exists a map $\delta: F \lrt G$ such that $i_{\la}(\delta)=\gamma$. Therefore $\Theta$ is full.

To prove that $\Theta$ is dense, pick $F  \in ((R\text{-}\underline{{\rm mod}})^{\rm op}, \mathcal{A}{\rm b})$. Consider a projective presentation
\begin{equation*}
\begin{tikzcd}
( - , \underline{M_1}) \rar{(-, \underline{f})} & (-, \underline{M_2}) \rar & F \rar & 0
\end{tikzcd}
\end{equation*}
of $F$ in $((R\text{-}\underline{{\rm mod}})^{\rm op}, \mathcal{A}{\rm b})$ such that $f: M_1 \lrt M_2$ is an epimorphism. Let $i: M_1\rt I$ be the injective envelop of $M_1.$ Consider the object $[f~~i]^{\rm{t}}:M_1 \rt M_2\oplus I$ in $\CS(R\mbox{-}{\rm Mod}).$ We claim that $\Theta([f~~i]^{\rm{t}})=F$. Set $N:={\rm Cok}[f~~i]^{\rm{t}}$. We can construct the  commutative diagram
\begin{equation}\label{eq. ThemTheta1}
\xymatrix{P \ar[r] \ar[d]^g & Q \ar[r] \ar[d]^h & N\ar@{=}[d]\ar[r] & 0 \\   M_1 \ar[r] \ar@{=}[d] &  M_2\oplus I \ar[r] \ar[d] & N \ar[r]  & 0 \\  M_1 \ar[r]^f &  M_2 \ar[r] & 0 &  }
\end{equation}
with exact rows, where $g, h$ are epimorphism and $P, Q$ are projectives. Set
\[G:={\rm Cok}(( - , M_1) \rt ( - , M_2 \oplus I)).\]

Hence $\nu(G)={\rm Cok}(M_1 \rt M_2\oplus I)=N.$ By applying the embedding $\Upsilon$ to Diagram \eqref{eq. ThemTheta1}, we obtain the commutative diagram
\[\xymatrix{( - , P) \ar[r] \ar[d] & ( - , Q) \ar[r] \ar[d] & L(\nu(G))\ar[r] \ar[d]^d & 0 \\  ( - , M_1) \ar[r] \ar[d] & ( - , M_2\oplus I) \ar[r] \ar[d] & G \ar[r] \ar[d] & 0 \\  ( - , \underline{M_1}) \ar[r] & ( - ,\underline{M_2}) \ar[r] & F \ar[r]\ar[d] & 0\\  &  & 0 &  }\]
with exact rows. Indeed, $\epsilon_G=d$, and consequently, $\Theta(G)=F$. So $\Theta$ is dense.

Finally, we show that $\Theta$ is objective. By definition, $\Theta$ is the composition of the following functors
\begin{equation*}
\begin{tikzcd}
\CS(R\mbox{-}{\rm Mod}) \rar{\Sigma} & \CF^{\leq 1} \rar{i_{\la}} & ((R\text{-}\underline{{\rm mod}})^{\rm op}, \mathcal{A}{\rm b}).
\end{tikzcd}
\end{equation*}

Since by Proposition \ref{flatdimension}, $\Sigma$ is objective, by \ref{ObjectiveFunctor}, we just need to prove that the restricted functor $i_{\la}$ is objective. To this end, let $\delta: F \rt G$ be a morphism in $\CF^{\leq 1}$ such that $i_{\la}(\delta)=0$. Hence we have the commutative diagram
\[\xymatrix{L(\nu(F)) \ar[r]^{\epsilon_F}  & F \ar[r] \ar[d]^{\delta} & i_{\la}(F) \ar[r] \ar[d]^{0}  & 0 \\ L(\nu(G)) \ar[r]^{\epsilon_G} & G \ar[r] & i_{\la}(G)\ar[r] & 0.}\]
Since the lower row is exact, $\delta$ factors through the functor $L(\nu(G))$. But by definition of $i_{\la}$, $L(\nu(G))$ is a kernel object of $i_{\la}$. This completes the proof.
\end{proof}

Because of the similarity in the arguments, we record the following corollary without proof.

\begin{corollary}
With the notations and assumptions as in Theorem \ref{Theta-Theorem}, there exists an equivalence
\[\CS(R\mbox{-}{\rm Mod})/\CX \simeq ((R\text{-}\underline{{\rm mod}})^{\rm op}, \mathcal{A}{\rm b}),\]
where $\CX$  is the subcategory of $\CS(R\mbox{-}{\rm Mod}) $ generated by the objects of the form $(M\st{f}\rt P)$, where $f$ is a monomorphism, $M$ is an arbitrary and $P$ is a projective module and, objects of the form $(N\st{1}\rt N)$, where $N$ runs through all modules.
\end{corollary}

\subsection{Covariant version}
Recall from \ref{Functor Gamma prime} the functor
\[\begin{array}{lll}
\Sigma': \CH(R\mbox{-}{\rm Mod}) \lrt  (R\text{-}{\rm mod}, \mathcal{A}{\rm b}), \vspace{0.1cm}\\
 \ \ \ \ \ \ \  (X\st{f}\rt Y) \  \mapsto  {\rm Ker}(t( X) \st{t(f)}\lrt  t(Y))
\end{array}\]
where the functor $t$ is defined in \eqref{functor t}.

\begin{definition}
The functor $\Im:\CF(R\mbox{-}{\rm Mod}) \rt (\underline{{\rm mod}}\mbox{-}R, \CA{\rm b})$ is defined as the composition of the functors
\begin{equation*}
\begin{tikzcd}
\CF(R\mbox{-}{\rm Mod})  \ar[r,hook] & \CH(R\mbox{-}{\rm Mod}) \rar{\Sigma'} & ({\rm mod}\mbox{-}R, \mathcal{A}{\rm b}) \rar{j_{\rho}} & (\underline{{\rm mod}}\mbox{-}R, \mathcal{A}{\rm b})
\end{tikzcd}
\end{equation*}
where $j_{\rho}$ is defined in \ref{Functor J rho}.
\end{definition}

Let $(Y\st{g}\rt Z)$ be an object in $\CF(R\mbox{-}{\rm Mod})$. So we have the short exact sequence
\begin{equation*}
\begin{tikzcd}
0 \rar & X \rar{f} & Y \rar{g} & Z \rar & 0,
\end{tikzcd}
\end{equation*}
where $X=\Ker g$. Let $X \st{i}\rt I$, be a monomorphism, where $I$ is an injective module. There exists a morphism $s: Y \rt I$ such that $sf=i$. This induces the epimorphism
\begin{equation*}
\begin{tikzcd}
Y \rar{\left[\begin{smallmatrix} g \\ s \end{smallmatrix}\right]} &  Z\oplus I \rar & 0.
\end{tikzcd}
\end{equation*}

Apply the functor $t$ of \ref{functor t} to get the exact sequence
\begin{equation*}
\begin{tikzcd}
0 \rar & G \rar &  - \otimes Y \rar{-\otimes\left[\begin{smallmatrix} g \\ s \end{smallmatrix}\right]} &  -\otimes (Z\oplus I)) \rar & 0.
\end{tikzcd}
\end{equation*}

By definition
\[\Im(Y \st{g}{\rt} Z) = G.\]

It is routine to check that this definition is independent of the choice of the injective module $I$ as well as the morphism $s$ and to check that $G$ is indeed an object of
$(\underline{{\rm mod}}\mbox{-}R, \mathcal{A}{\rm b}).$

\begin{remark}\label{Rem, injec-Fuc-stable-contarvariant}
By \cite[Page 379]{He}, the subcategory $(\underline{{\rm mod}}\mbox{-}R, \CA{\rm b})$ is a hereditary torsion class of $({\rm mod}\mbox{-}R, \CA{\rm b})$. If $F \in ({\rm mod}\mbox{-}R, \CA{\rm b})$, we will denote the $(\underline{{\rm mod}}\mbox{-}R, \CA{\rm b})$-torsion subobject of $F$ by ${\rm tor}(F)$. Let $M$ be an $R$-module and let $e\colon M\rt E$ be the injective envelope of $M.$  By \cite[Proposition 28]{He}, it is known that
\begin{equation*}
\begin{tikzcd}
{\rm tor}(t(M))={\rm tor}(-\otimes M)={\rm Ker}(t(M)  \rar{t(e)} & t(E)).
\end{tikzcd}
\end{equation*}
Let $F \in (\underline{{\rm mod}}\mbox{-}R, \CA{\rm b})$. Consider a copresentation of $F$ by fp-injective objects in $({\rm mod}\mbox{-}R, \CA{\rm b})$
\begin{equation*}
\begin{tikzcd}
0 \rar & F \rar & - \otimes M \rar{- \otimes f} & - \otimes N.
\end{tikzcd}
\end{equation*}
Since the torsion class $(\underline{{\rm mod}}\mbox{-}R, \CA{\rm b})$ is hereditary, the torsion functor ${\rm tor}$ is left exact. On the other hand, since ${\rm tor} (F)=F$, applying the functor {\rm tor} yields the exact sequence
\[\xymatrix{ 0 \ar[r] & F \ar[r] & {\rm tor}( - \otimes M) \ar[rr]^{{\rm tor}( - \otimes f)} && {\rm tor}( - \otimes N).}\]
We use this sequence in the proof of the next theorem.
\end{remark}

\begin{theorem}\label{SeconTheoemcovariant}
Suppose that $R$ is a left pure-semisimple and Quasi-Frobenius ring. Then, the functor
\begin{equation*}
\begin{tikzcd}
\Im:\CF(R\mbox{-}{\rm Mod}) \rar & (\underline{{\rm mod}}\mbox{-}R, \CA{\rm b})
\end{tikzcd}
\end{equation*}
is full, dense and objective.
\end{theorem}

\begin{proof}
First, we show that $\Im$ is full. Let $\CI^{\leq 1}$ denote the subcategory of $({\rm mod}\mbox{-}R, \CA{\rm b})$ consisting of all functors  $F$ which admits an injective coresolution
\begin{equation*}
\begin{tikzcd}
0 \rar & F\rar & - \otimes M \rar & -\otimes N \rar & 0,
\end{tikzcd}
\end{equation*}
where $M, N$ are pure-injective modules. Since $R$ is a left pure-semisimple ring, $\CF(R\mbox{-}{\rm Mod}) =\CF_0(R\mbox{-}{\rm Pinj}) $.  We can present $\Im$ as the composite of the following functors
\begin{equation}\label{eq. Theorem second Covar}
\begin{tikzcd}
\CF(R\mbox{-}{\rm Mod}) \rar{\Sigma'} & \CI^{\leq 1} \rar{j_{\rho}} & (\underline{{\rm mod}}\mbox{-}R, \CA{\rm b}).
\end{tikzcd}
\end{equation}

For every object $(Y\st{g}\rt Z) \in \CF(R\mbox{-}{\rm Mod})$,  we have the fp-injective coresolution
\begin{equation}\label{eq. Theoem.theta}
\begin{tikzcd}
0 \rar & \Sigma'(g) \rar & -\otimes Y \rar & -\otimes Z \rar & 0
\end{tikzcd}
\end{equation}
such that $-\otimes Y$ is injective. This implies that any morphism between $\Sigma'(g)$ and $\Sigma'(g')$, where $g': Y' \rt Z'$, can be lifted to a chain map between the corresponding  fp-injective coresolutions. In view of the embedding $t$, we obtain a morphism between $g$ and $g'$. Hence, the functor $\Sigma'$ in  \eqref{eq. Theorem second Covar} is full. Therefore, to prove the fullness of  the functor $\Im$, we just need to show that $j_{\rho}$ is full on $\Sigma'(\CF(R\mbox{-}{\rm Mod})) = \CI^{\leq 1}$, as stated in \eqref{eq. Theorem second Covar}.

Let $F$ and $G$ be functors in ${\rm fp}^{\leq 1}$ and $\delta: j_{\rho}(F)\rt j_{\rho}(G)$ be a morphism in $(\underline{{\rm mod}}\mbox{-}R, \CA{\rm b})$.
We have the following diagram
\[\xymatrix{0\ar[r] &j_{\rho}(F)\ar[r]\ar[d]^{\delta}  & F \ar[r]^>>>>>{\eta_F} \ar@{.>}[d]^{\sigma} & R_0(\va(F))    \\ 0 \ar[r] & j_{\rho}(G) \ar[r] & G \ar[r]^>>>>>{\eta_G} & R_0(\va(G)),}\]
with exact rows and the units $\eta_F$ and $\eta_G$. In the following, we prove the existence of $\sigma$. Let $D$ denotes the image  of $\eta_F$. We show that $\Ext^1(D, G)=0,$ where $\Ext^1(D, G)$ denote the abelian group of equivalence classes of extensions of $G$ by $D$ in the category $(\mmod R, \mathcal{A}{\rm b})$.  This implies that there exists $\sigma: F \rt G$ such that  $j_{\rho}(\sigma)=\delta$, as desired.
To show $\Ext^1(D, G)=0$ we apply the isomorphism
\[\Ext^1_{}(D, G)\simeq \Hom_{\mathbb{K}}(\mathbf{I}_D, \mathbf{I}_G[1]),\]
where $\mathbb{K}$ denotes the homotopy category of complexes of functors of $(\mmod R, \mathcal{A}{\rm b})$, and $\mathbf{I}_D$ and $\mathbf{I}_G$ are deleted injective coresolutions of $D$ and $G$, respectively. Since $F \in \CI^{\leq 1}$, we have the injective coresolution
\begin{equation*}
\begin{tikzcd}
0 \rar & F \rar & -\otimes Y \rar{-\otimes g} \rar & -\otimes Z\rar & 0.
\end{tikzcd}
\end{equation*}
By definition, $R_0(\va(F))$ can be embedded into an injective functor $-\otimes I$, where $I$ is injective module.  The embedding gives rise to an injective coresolution of $D$ as in the following
\begin{equation*}
\begin{tikzcd}
0 \rar & D \rar & -\otimes I \rar &  -\otimes M \rar & -\otimes N \rar & 0.
\end{tikzcd}
\end{equation*}
Now consider a chain map
\[\xymatrix{0 \ar[r]&0 \ar[r] \ar[d]&-\otimes I\ar[r] \ar[d]^{-\otimes h} & -\otimes M \ar[r] \ar[d] & -\otimes N \ar[r] \ar[d]  & 0 \\ 	0\ar[r] & -\otimes Y \ar[r]^{-\otimes g}&	-\otimes Z \ar[r] & 0\ar[r] & 0 \ar[r] & 0}\]
from $\mathbf{I}_D$ to $\mathbf{I}_G[1]$. Since $R$ is Quasi-Frobenius, $I$ is a projective module. Since $g$ is an epimorphism, we can deduce $-\otimes h$ factors through the morphism $-\otimes g$. This, in turn, implies that the chain map is null-homotopic. Hence $\Hom_{\mathbb{K}}(\mathbf{I}_D, \mathbf{I}_G[1])=0$, so the result.

In order to show that $\Im$ is dense, take $F  \in (\underline{{\rm mod}}\mbox{-}R, \CA{\rm b})$. In view of Remark \ref{Rem, injec-Fuc-stable-contarvariant}, we have the following  injective copresentation  of $F$ in   $(\underline{{\rm mod}}\mbox{-}R, \CA{\rm b})$
\[\xymatrix{0 \ar[r] & F \ar[r] & {\rm tor}(-\otimes M_1) \ar[rr]^{{\rm tor} (-\otimes f)} & & {\rm tor}(-\otimes M_2).}\]

As $F$ vanishes on projective modules, we deduce that $f$ is an injection. Let $p: P\rt M_2$ be the projective precover of $M_2,$ i.e. $P$ is projective and $p$ is a projection. Consider the object $((M_1\oplus P)\st{[f~~p]}\lrt M_2).$ We claim that $\Im([f~~p])=F$. Set $N:={\rm Ker}([f~~p])$. We have the commutative diagram
\begin{equation}\label{eq. ThemTheta2}
\xymatrix{&0 \ar[r]  & M_1 \ar[r]^f \ar[d]_{[1~~0]^t} & M_2\ar@{=}[d]\ar[r] & 0 \\  0 \ar[r] &N\ar[r] \ar@{=}[d] &  M_1\oplus P \ar[r]^{[f~~p]} \ar[d]^g & M_2 \ar[r]\ar[d]^h  & 0 \\ 0 \ar[r]& N\ar[r] &  I_1 \ar[r] & I_2&  }
\end{equation}
 with exact rows, where $g$ and $h$ are monomorphisms and $I_1$ and $I_2$ are injectives. Let
\[G:={\rm Ker}(-\otimes(M_1\oplus P) \rt -\otimes M_2 ).\]
\noindent
Hence $\va(G)={\rm Ker}(M_1\oplus P \rt M_2).$ By applying the tensor functor on Diagram \eqref{eq. ThemTheta2} and using the snake lemma, we obtain the commutative diagram
\begin{equation}
\xymatrix{  & 0\ar[d] &   & & \\  0 \ar[r]&F \ar[r] \ar[d] & {\rm tor}(-\otimes M_1) \ar[r]^{{\rm tor}(-\otimes f)} \ar[d] & {\rm tor}(-\otimes M_2)\ar[d] &  \\  0 \ar[r] &G\ar[r] \ar[d]^l &  -\otimes (M_1\oplus P) \ar[r]^{-\otimes [f~~p]} \ar[d] & -\otimes M_2\ar[d]   &  \\ 0 \ar[r]& R_0(\va(G))\ar[r] &  -\otimes I_1 \ar[r] & -\otimes I_2&  }
\end{equation}
 with exact rows. Indeed, $\eta_G=l$, and consequently, $\Im(G)=F$. Hence $\Im$ is dense.

Finally, we show that $\Im$ is objective.  According to \eqref{eq. Theorem second Covar}, $\Im$ is the composition of $j_{\rho}$ and $\Sigma'.$
Since by Proposition \ref{fp-inj-dimension}, $\Sigma'$ is  objective, by \ref{ObjectiveFunctor}, we need to prove that the restricted functor $j_{\rho}$ is objective. To this end, let $\delta: F \rt G$ be a morphism in $\CF^{\leq 1}$ such that $j_{\rho}(\delta)=0$. Hence we have the commutative diagram
\[\xymatrix{0\ar[r] &j_{\rho}(F)\ar[r]\ar[d]^{0}  & F \ar[r]^>>>{\eta_F} \ar[d]^{\delta} & R_0(\va(F))    \\ 0 \ar[r] & j_{\rho}(G) \ar[r] & G \ar[r]^>>>{\eta_G} & R_0(\va(G)),}\]

Since the upper row is exact, $\delta$ factors through the functor $R_0(\va(F))$. But by the definition of $J_{\rho}$, $R_0(\va(F))$ is a kernel object of $j_{\rho}$. Hence the proof is complete.
\end{proof}

Following corollary is an immediate consequence of the above proposition. We omit the proof.

\begin{corollary}
With the notations and assumptions of Theorem \ref{SeconTheoemcovariant}, there exists an equivalence
\[\CF(R\mbox{-}{\rm Mod})/\CY \simeq  (\underline{{\rm mod}}\mbox{-}R, \CA{\rm b}) \]
where $\CY$  is the subcategory of $\CF(R\mbox{-}{\rm Mod}) $ generated by the objects of the form $(P\st{p}\rt M)$, where $p$ is an epimorphism, $P$ is a projective and $M$ is an arbitrary module and, objects of the form $(N\st{1}\rt N)$, where $N$ runs through all modules. The restriction of this equivalence to $R\mbox{-}{\rm mod}$ induces an equivalence
\[\CF(R\mbox{-}{\rm mod})/\CY \simeq {\rm fp} (\underline{{\rm mod}}\mbox{-}R, \CA{\rm b}). \]
\end{corollary}

\s Let $R$ be a ring that has a self-duality $D$ between $\mmod R$ and $\mmod R^{\op}$. It is shown in \cite[Proposition 5.5]{K1} that there is an equivalence
\[\CA: \underline{{\rm Mod}}\text{-}R \st{\simeq}\lrt \overline{{\rm Mod}}\text{-}R,\]
which is called the Auslander-Reiten translation. Krause showed that this functor coincides on the finitely presented level with the dual of transpose functor
\[D{\rm Tr}: \underline{{\rm mod}}\text{-}R  \st{\simeq}\lrt \overline{{\rm mod}}\text{-}R.\]

In the following, we show that our results can be applied to get a similar equivalence for certain algebras.

Let $X$ be an $R$ module. Since $R$ is QF, there exists a short exact sequence
\begin{equation*}
\begin{tikzcd}
0 \rar & \Omega(X) \rar & P_X \rar & X \rar & 0,
\end{tikzcd}
\end{equation*}
where $P_X \rt X$ is the projective cover of $X$.

Assume that $\CX'$ is an additive subcategory of $\CS(R\mbox{-}{\rm Mod})$ generated by all objects of the forms $X \st{1}\lrt X, 0 \lrt X$ and $\Omega(X) \lrt P$, where $X$ ranges over all $R$-modules. Furthermore, let $\CY'$ be the additive subcategory of $\CF(R\mbox{-}{\rm Mod})$  generated by all objects of the forms $Y \st{1}\lrt Y, Y \lrt 0$ and $P_Y \lrt Y$, where $Y$ runs through all $R$-modules. It is obvious that ${\rm Cok}(\CX') = \CY'$ and $\Ker(\CY') = \CX'$.

On the other hand, it follows by definition that the functor $\Theta$ of Theorem \ref{Theta-Theorem} maps $0 \lrt X$ to $( - ,\underline{X})$, a projective object in $((R\text{-}\underline{{\rm mod}})^{\rm op}, \mathcal{A}{\rm b})$ and the map $\Im$ of Theorem \ref{SeconTheoemcovariant} maps $Y \lrt 0$ to $- \otimes \underline{Y}$, an injective object in $(\underline{{\rm mod}}\mbox{-}R, \CA{\rm b})$.

Now one can follow standard arguments to see that there exists the following sequence of equivalences
\[\begin{tikzcd}[column sep = large]
\underline{((R\text{-}\underline{{\rm mod}})^{\rm op}, \mathcal{A}{\rm b})} \rar{\simeq} & \CS(R\mbox{-}{\rm Mod})/\CX' \rar["{\rm Cok}", shift left]  &  \lar["{{\rm Ker}}", shift left]   \CF(R\mbox{-}{\rm Mod})/\CY' \rar{\simeq}  & \overline{(\underline{{\rm mod}}\mbox{-}R, \CA{\rm b})}.
\end{tikzcd} \]

Now if $\La$ is a selfinjective artin algebra of finite representation type, the above equivalences, in view of Corollary \ref{CorofinitRep}, implies the equivalence
\[\underline{{\rm Mod}}\text{-}\Gamma \st{\simeq}\lrt \overline{{\rm Mod}}\text{-}\Gamma,\]
where $\Ga$ is the stable Auslander algebra of $\La$.

\section{An equivalence}
Let ${\rm fp}((R\text{-}{\rm Mod})^{\op}, \CA{\rm b})$ denote the category of all finitely presented contravariant functors from $R\text{-}{\rm Mod}$ into the category of abelian groups. Similarly, let ${\rm fp}(R^{\op}\text{-}{\rm Mod}, \CA{\rm b})$, denote the category of all finitely presented covariant functors from $\Mod R$  into the category of abelian groups. These categories have been studied in depth by Auslander in \cite{Au}. Moreover, there are functors
\[ \CR: {\rm fp}((R\text{-}{\rm Mod})^{\op}, \CA{\rm b}) \lrt ((R\text{-}{\rm mod})^{\op}, \CA{\rm b}), \]
called the restriction functor, that maps any functor $F \in {\rm fp}((R\text{-}{\rm Mod})^{\op}, \CA{\rm b})$ to its restriction to $R\text{-}{\rm mod}$, and the functor
\[\mathcal{D}: (\Mod R, \CA{\rm b}) \lrt  (\mmod R, \CA{\rm b})\]
called the Auslander-Gruson-Jensen functor. This is an exact contravariant functor that maps the object $(M, -)$ to the tensor functor $M \otimes - $. For more details on the properties of this functor, see \cite{DR1}. In particular, it is shown that this functor has a fully faithful right adjoint and a fully faithful left adjoint \cite{DR1}.

In this section, we pursue such a study in the stable category version.

\s Let $\SV$, resp. $\CV$, denote the additive subcategory of $\CS(R\mbox{-}{\rm Mod})$, resp. $\CS(R\mbox{-}{\rm Pinj})$, generated by all objects of the forms $X \st{1}\lrt X$ and $0 \lrt X$, where $X$ ranges over all $R$-modules, resp. pure-injective $R$-modules. Similarly, assume that $\SU$, resp. $\CU$, is the additive subcategory of $\CF(R\mbox{-}{\rm Mod})$, resp. $\CF(R\mbox{-}{\rm Pinj})$, generated by all objects of the forms $X \st{1}\lrt X$ and $X \lrt 0$, where $X$ ranges over all $R$-modules, resp. pure-injective $R$-modules.

Since $\CV \subseteq \SV$ and $\CU \subseteq \SU$, the embeddings
\begin{equation*}
\begin{tikzcd}
\CS(R\mbox{-}{\rm Pinj}) \rar[hook] & \CS(R\mbox{-}{\rm Mod}) \ {\rm and} \ \CF(R\mbox{-}{\rm Pinj}))  \rar[hook] & \CF(R\mbox{-}{\rm Mod})
\end{tikzcd}
\end{equation*}
induces functors
\begin{equation*}
\begin{tikzcd}
\Delta_1: \CS(R\mbox{-}{\rm Pinj})/\CV \rar[hook] & \CS(R\mbox{-}{\rm Mod})/\SV \ {\rm and} \ \Delta_2: \CF(R\mbox{-}{\rm Pinj})/\CU  \rar[hook] & \CF(R\mbox{-}{\rm Mod})/\SU.
\end{tikzcd}
\end{equation*}

We show that $\Delta_1$ and $\Delta_2$ are full and faithful and hence are embedding.

\begin{lemma}
The functors $\Delta_1$ and $\Delta_2$ are full and faithful.
\end{lemma}

\begin{proof}
We prove that $\Delta_1$ is full and faithful. Similar, or rather, dual proof works to show that $\Delta_2$ is full and faithful. The fullness of $\Delta_1$ follows from definition. For the faithfullness, let
\[(\delta_1, \delta_2): (X \st{f}{\lrt} Y) \lrt (X' \st{f'}{\lrt} Y')\]
be a morphism in $\CS(R\mbox{-}{\rm Pinj})$ that factors in $\CS(R\mbox{-}{\rm Mod})$ through the object $(V_1 \lrt V_2)$ in $\SV$. We show that $(\delta_1, \delta_2)$ factors through an object in $\CV$. To this end, note that it is easy to check that $(X' \st{1}{\rt} X') \oplus (0 \lrt Y') $ with the appropriate morphisms provides a right $\SV$-approximation of $X' \rt Y'$ as follows
\begin{equation*}
\begin{tikzcd}
X' \oplus 0 \rar{[1\ 0]} \dar{[1 \ 0]} & X' \dar{f'} \\ X' \oplus Y' \rar{[f' \ 1]} & Y'
\end{tikzcd}
\end{equation*}
So $(V_1 \lrt V_2) \lrt (X' \st{f'}{\lrt} Y')$ factors through $(0 \lrt Y') \oplus (X' \st{1}{\rt} X')$, that belongs to $\CV$. Hence $(\delta_1, \delta_2)$ factors through an object in $\CV$. Therefore $\Delta_1$ is faithful.
\end{proof}

We apply this lemma to prove the following application. By Theorems 3.3 and 3.4 of \cite{Ha1} there are equivalences
 \begin{equation*}\label{Obs-1-1}
\frac{\CS(R\mbox{-}{\rm Mod})}{\SV} \lrt {\rm fp}((R\mbox{-}\underline{{\rm Mod}})^{\rm op}, \CA{\rm b}),
\end{equation*}
and
\begin{equation*}\label{Obs-1-2}
\frac{\CF(R\mbox{-}{\rm Mod})}{\SU} \lrt ({\rm fp}(R\mbox{-}\overline{{\rm Mod}}, \CA{\rm b}))^{\rm op},
\end{equation*}
where $\SV$ and $\SU$ are introduced above.

On the other hand, by Corollaries \ref{Cor. Psi} and \ref{Cor.Phi}, there exists equivalences
\[\Psi':\CS(R\mbox{-}{\rm Pinj})/\CV\st{\sim} \lrt ((R\text{-}\underline{{\rm mod}})^{\rm op}, \mathcal{A}{\rm b}),\]
and
\[\Phi':\CF(R\mbox{-}{\rm Pinj})/\CU\st{\sim} \lrt (\underline{{\rm mod}}\mbox{-}R, \CA{\rm b}).\]

Putting all together, we get the commutative diagram

\begin{equation*}
\begin{tikzcd}[column sep=small]
& & & \CS(R\mbox{-}{\rm Mod})/\SV \dlar \ar{rr}  && \CF(R\mbox{-}{\rm Mod})/\SU \dlar  \\
&&    {\rm fp}((R\mbox{-}\underline{{\rm Mod}})^{\rm op}, \CA{\rm b})  \ar{rr}  &&  ({\rm fp}(R\mbox{-}\overline{{\rm Mod}}, \CA{\rm b}))^{\rm op} \\
&&& \CS(R\mbox{-}{\rm Pinj})/\CV \dlar \ar{rr}  \ar[hook]{uu} && \CF(R\mbox{-}{\rm Pinj})/\CU \dlar \ar[hook]{uu}\\
&&  ((R\text{-}\underline{{\rm mod}})^{\rm op}, \mathcal{A}{\rm b}) \ar{rr} \ar[hook]{uu} && (\underline{{\rm mod}}\mbox{-}R, \CA{\rm b}) \ar[hook]{uu}.
\end{tikzcd}
\end{equation*}

This, in turn, induces the equivalence
\begin{equation*}
\begin{tikzcd}
 ((R\text{-}\underline{{\rm mod}})^{\rm op}, \mathcal{A}{\rm b}) \rar{\simeq} & (\underline{{\rm mod}}\mbox{-}R, \CA{\rm b})
\end{tikzcd}
\end{equation*}

This equivalence already have been proved by Herzog in \cite[Page 382]{He}. The left hand square of the diagram, in turn, provides a non-obvious embedding from $((R\text{-}\underline{{\rm mod}})^{\rm op}, \mathcal{A}{\rm b})$ to ${\rm fp}((R\mbox{-}\underline{{\rm Mod}})^{\rm op}, \CA{\rm b})$, while the right hand square provides a similar result for covariant version. If we further assume that $R$ is a left pure semisimple ring, then $R\mbox{-}{\rm Mod}=R\mbox{-}{\rm Pinj}$ and so the vertical maps $\Delta_1$ and $\Delta_2$ are equality. This, in particular, implies that the vertical maps in front are equivalences of categories, that is, we have equivalences
\begin{equation*}
\begin{tikzcd}
 {\rm fp}((R\mbox{-}\underline{{\rm Mod}})^{\rm op}, \CA{\rm b})  \rar{\simeq} & ((R\text{-}\underline{{\rm mod}})^{\rm op}, \mathcal{A}{\rm b})
\end{tikzcd}
\end{equation*}
and
\begin{equation*}
\begin{tikzcd}
 ({\rm fp}(R\mbox{-}\overline{{\rm Mod}}, \CA{\rm b})^{\rm op})  \rar{\simeq} & (\underline{{\rm mod}}\mbox{-}R, \CA{\rm b}).
\end{tikzcd}
\end{equation*}

\section*{Acknowledgments}
The work of the second author is based on research funded by Iran National Science Foundation (INSF) under project No. 4001480. The research of the second and the last author is partially supported by the Belt and Road Innovative Talents Exchange Foreign Experts project (Grant No. DL2023014002L). The research of the last author is supported by the National Natural Science Foundation of China (Grant No. 12101316).

\end{document}